\documentclass[11pt,reqno]{amsart}
\usepackage{etex}
\usepackage{amsmath,amsthm,amsfonts,amssymb,amscd,flafter,pinlabel, booktabs}
\usepackage[mathscr]{eucal}
\usepackage{graphics}
 \usepackage[all,knot,arc]{xy}
 \usepackage[normalem]{ulem}

\usepackage[usenames,dvipsnames]{xcolor}
\usepackage{subfigure}
\usepackage{graphicx}
 \usepackage{tabu}

\usepackage{bbm}
\usepackage{mathtools}

\usepackage{pgf}
\usepackage{tikz}
\usepackage{tikz-cd}
\usetikzlibrary{arrows,automata}
\usepackage[latin1]{inputenc}

\allowdisplaybreaks

\usepackage{cite}
\usepackage[final, colorlinks=true,linkcolor=blue,citecolor=blue,urlcolor=blue]{hyperref}

\tikzset{cdlabel/.style={above,sloped,
    execute at begin node=$\scriptstyle,execute at end node=$}}    
\tikzset{al/.style={->, bend right=45, thick}}
\tikzset{ar/.style={->, bend left=45, thick}}

\normalfont\upshape

\usepackage[lmargin=1in,rmargin=1in,tmargin=1in,bmargin=1in]{geometry}

\theoremstyle{definition}
\newtheorem{definition}[equation]{Definition}
\newtheorem*{remark*}{Remark}
\newtheorem*{example*}{Example}
\newtheorem{example}[equation]{Example}

\theoremstyle{plain}

\newtheorem{proposition}[equation]{Proposition}

\numberwithin{equation}{section}
\numberwithin{figure}{section}
\numberwithin{figure}{section}

\newtheorem{ourtheorem}{Theorem} 
\newtheorem{ourcorollary}[ourtheorem]{Corollary}

\newcommand\fol{\mathcal{F}}

\newcommand{\R}{\ensuremath{\mathbb{R}}}
\newcommand{\Z}{\ensuremath{\mathbb{Z}}}

\newcommand\alphas{\mbox{\boldmath$\alpha$}}
\newcommand\betas{\mbox{\boldmath$\beta$}}

\newcommand\HD{\mathcal H}

\newcommand\pit{{\widetilde{\pi}}}
\newcommand\bsGamma{\mathsf{\Gamma}}
\newcommand{\ainf}{\mathcal{A}_\infty}

\newcommand{\cfdt}{\widetilde{\cfd}}
\newcommand{\cfat}{\widetilde{\cfa}}

\newcommand{\bdy}{\partial}

\newcommand{\cala}{\mathcal{A}}

\DeclareMathOperator{\Int}{Int}

\DeclareMathOperator{\id}{id}

\newcommand{\HH}{\mathcal{H}}

\newcommand{\zz}{\mathcal Z}

\newcommand{\xxx}{\mathbf{x}}

\newcommand{\cf}{\mathit{CF}}
\newcommand{\hf}{\mathit{HF}}
\newcommand{\cfd}{\mathit{CFD}}
\newcommand{\cfa}{\mathit{CFA}}
\newcommand{\bsd}{\mathit{BSD}}
\newcommand{\bsa}{\mathit{BSA}}

\newcommand{\sfh}{\mathit{SFH}}

\newcommand{\cfhat}{\widehat{\cf}}
\newcommand{\hfhat}{\widehat{\hf}}
\newcommand{\hftilde}{\widetilde{\hf}}
\newcommand{\cftilde}{\widetilde{\cf}}
\newcommand{\cfdhat}{\widehat{\cfd}}
\newcommand{\cfahat}{\widehat{\cfa}}
\newcommand{\bsdhat}{\widehat{\bsd}}
\newcommand{\bsahat}{\widehat{\bsa}}

\newcommand{\brho}{\boldsymbol\rho}

\begin{document}

\title[{A friendly introduction to the bordered contact invariant}]{A friendly introduction to the bordered contact invariant}

\author[A. Alishahi]{Akram Alishahi}
\address {Department of Mathematics, University of Georgia\\ Athens, GA 30606}
\email {akram.alishahi@uga.edu}
\urladdr{\href{https://akramalishahi.github.io/}{https://akramalishahi.github.io/}}
\thanks{AA was supported by NSF Grant DMS-2019396.}
\author[J. Licata]{Joan Licata}
\address {Mathematical Sciences Institute, The Australian National University\\ Canberra, Australia}
\email {joan.licata@anu.edu.au}
\urladdr{\href{https://sites.google.com/view/joanlicata/home}{https://sites.google.com/view/joanlicata/home}}
\author[I. Petkova]{Ina Petkova}
\address {Department of Mathematics, Dartmouth College\\ Hanover, NH 03755}
\email {ina.petkova@dartmouth.edu}
\urladdr{\href{http://math.dartmouth.edu/~ina}{http://math.dartmouth.edu/~ina}}
\author[V. V\'ertesi]{Vera V\'ertesi}
\address {Faculty of Mathematics, University of Vienna}
\email {vera.vertesi@univie.ac.at}
\urladdr{\href{http://mat.univie.ac.at/~vertesi}{www.mat.univie.ac.at/~vertesi}}
\thanks{VV was supported by the ANR grant ``Quantum topology and contact geometry''}


\begin{abstract}
We give a short introduction to the contact invariant in bordered  Floer homology defined by F{\"o}ldv{\'a}ri, Hendricks, and the authors. This article surveys  the contact geometry  required to understand the new invariant but  assumes some familiarity  with bordered Heegaard Floer invariants.   The input for the construction is a special class of foliated open books, which are introduced carefully and with multiple examples. We discuss how a foliated open book may be constructed from an open book for a closed manifold, and how it may  be modified to ensure compatibility  with the  contact bordered invariant.   As an  application of these techniques, we  give a ``local proof"  of the vanishing of the contact invariant for overtwisted structures  in the  form  of an explicit bordered computation.
\end{abstract}

\maketitle



\section{Introduction}
\label{sec:intro}

Contact geometry, often pitched as the odd-dimensional complement to symplectic geometry, considers a $(2k+1)$-dimensional manifold equipped with some additional structure.  In dimension three -- where we reside henceforth -- this extra data is a nowhere-integrable plane field called a contact structure.   Adding this extra data prompts interesting new questions, but one of the most intriguing features of the subject is that this ``extra'' data also offers insight into topological structure apparently unrelated to plane fields at all.  Two notable examples are the role of contact geometry in the proof of the property P conjecture \cite{km-p} and  the proofs that knot Floer homology detects knot genus \cite{hfkg} and fiberedness \cite{pgf, ynf}.

Contact structures themselves split into two mutually exclusive types, known as \emph{tight}  and \emph{overtwisted}.   Overtwisted structures are determined by homotopical data, and so are easy to understand. In contrast, tight contact structures are more mysterious: some, but not all, tight contact structures arise naturally as the boundary of symplectic manifolds, and tight contact structures do not satisfy an $h$-principle. Many existence and classification questions for tight contact structures remain open, but significant progress has been made since the advent of Heegaard Floer homology in the early 2000s and the subsequent development of Floer-theoretic contact invariants.  

Like other Heegaard Floer invariants, the input data for these constructions is a Heegaard diagram for the three-manifold, but in this setting, the Heegaard diagram is induced from an \textit{open book decomposition}, a topological decomposition of a three-manifold that captures the additional data of an equivalence class of contact structure.  Ozsv\'ath and Szab\'o defined the first Heegaard Floer invariant of closed contact three-manifolds in \cite{oszc}. Given a closed, contact manifold $(M,\xi)$, this invariant is a class $c(\xi)$ in the Heegaard Floer homology $\hfhat(-M)$.   In \cite{HKM09_HF}, Honda, Kazez, and Mati\'c gave an alternative description of $c(\xi)$, again using open books. This ``contact class"  gives information about overtwistedness:    if $\xi$ is overtwisted, then $c(\xi) = 0$, whereas  if $\xi$ is Stein fillable, then $c(\xi)\neq 0$ \cite{oszc}.  The contact class was used in the knot Floer homology proofs noted  above, and also to distinguish notions of fillability: Ghiggini used it to construct examples of strongly symplectically fillable contact three-manifolds which do not have Stein fillings \cite{ghfill}.

In this paper, we discuss a recent extension of the contact class to three-manifolds with boundary. Namely, in \cite{afhlpv}, a contact invariant was defined in the bordered sutured  Floer homology of a foliated contact three-manifold  $(M,\xi, \fol)$, which is  a contact manifold  
with a certain type of singular foliation on the boundary. We associate to a foliated  contact three-manifold   a  bordered sutured manifold $(M,  \bsGamma,\mathcal{Z})$.  
The resulting sutures are particularly simple, so one can think of $(M,  \bsGamma,\mathcal{Z})$ as a bordered manifold $(M, \mathcal{Z})$ of a type slightly more general than in \cite{LOTbook}. 
Below, we rephrase the main results of \cite{afhlpv}, translating from ``bordered sutured" to ``multipointed" language. Section~\ref{sec:mpbordered} explores the correspondence between these two viewpoints in more detail.

Using a special decomposition of $(M,\xi, \fol)$ called a \emph{sorted foliated open book}, one can construct an admissible multipointed bordered Heegaard diagram for the manifold $(M, \mathcal{Z})$ and identify a preferred generator. This preferred generator is an invariant of the contact structure. 

 \begin{ourtheorem}[cf.~{\cite[Theorem 1]{afhlpv}}]\label{thm:ca-cd}
Let  $(M,\xi, \fol)$  be a foliated  contact three-manifold with associated  bordered manifold $(M, \mathcal{Z})$. Then there are invariants $c_D(M,\xi, \fol)$ and $c_A(M,\xi, \fol)$ of the contact structure which are well defined homotopy equivalence classes in the multipointed bordered  Floer homologies  $\cfdt(-M, \overline{\zz})$ and  $\cfat(-M, \overline{\zz})$, respectively. 
 \end{ourtheorem} 
 
Furthermore, this generator vanishes for overtwisted manifolds, in the following sense. 
 
 \begin{ourtheorem}[cf.~{\cite[Corollary 4]{afhlpv}}]
\label{thm:ot} If $(M, \xi, \mathcal{F})$ is overtwisted, then the classes $c_D(M,\xi, \fol)$ and $c_A(M, \xi, \fol)$ are zero in $H_{\ast}(\cfdt(-M, \overline{\zz}))$ and $H_{\ast}(\cfat(-M, \overline{\zz}))$, respectively. 
 \end{ourtheorem} 
 
 Given a pair of foliated contact three-manifolds $(M^L, \xi^L, \fol^L)$ and $(M^R, \xi^R, \fol^R)$  whose  boundaries agree in an appropriate sense, there is a natural way to glue them to obtain a closed contact three-manifold $(M, \xi)$. 
The contact invariants of the two foliated contact three-manifolds pair  to recover the contact invariant  of $(M, \xi)$.

\begin{ourtheorem}[cf.~{\cite[Theorem 2]{afhlpv}}]
\label{thm:glue}
The tensor product $c_A(M^L,\xi^L,\fol^L)\boxtimes c_D(M^R,\xi^R,\fol^R)$ recovers the contact invariant $c(M,\xi)$.
 \end{ourtheorem}

This paper offers a hands-on  introduction to the bordered contact invariant, favoring geometric intuition over the formal proofs that may be found in \cite{LV} and \cite{afhlpv}. We assume minimal background in contact geometry, so Section~\ref{sec:prelim}  focuses on understanding contact structures via  characteristic foliations. Section~\ref{sec:mpbordered} introduces multipointed bordered Floer homology as a special case of bordered sutured Floer homology, laying the groundwork for a simplified description of the construction of the bordered contact invariant.  Section~\ref{sec:fob} discusses  open books, reviewing the  classical case for closed manifolds  before introducing  foliated open books for manifolds  with boundary.   After exploring some topological examples we define the contact  structure supported by a foliated open book.     We also define the technical condition ``sorted'' for a foliated open book and explain how it may be achieved by  stabilization preserving the supported contact structure.  We illustrate this in a  carefully chosen example of a foliated open book for a neighborhood of an overtwisted disk. In Section~\ref{sec:hd} we describe how to construct a Heegaard diagram from a sorted foliated open book and define an associated generator that represents the contact invariant. Finally, in Section~\ref{sec:local} we extend the earlier example to construct a Heegaard diagram for an overtwisted ball.  A local computation, in conjunction with Theorem~\ref{thm:glue}, then recovers the following vanishing result: 

\begin{ourcorollary}[\cite{oszc}]
\label{cor:ot}
Let $(M, \xi)$  be a closed contact three-manifold. If $\xi$ is overtwisted, then $c(\xi) = 0$.
\end{ourcorollary} 
 
  Note that  \cite{hkm09} establishes the vanishing of the sutured contact class for a neighborhood of an overtwisted disk.   The TQFT gluing map from \cite{hkm08} then yields a sutured argument that $c(\xi)$ vanishes for overtwisted closed manifolds. Our local construction explicitly constructs the  ``contact compatible'' layer needed in the sutured setting, giving a bordered counterpart to the argument.

\subsection*{Acknowledgements} 
We are grateful to BIRS for hosting the workshop \emph{Interactions of gauge theory with contact and symplectic topology in dimensions 3 and 4}. We also appreciate the helpful feedback provided by the  referee on an earlier version of this paper.


\section{Contact manifolds and surfaces}\label{sec:prelim}
A key aim of this paper is to render more accessible a new invariant in  bordered sutured Floer homology, but we'd like to start with a discussion of what this is an invariant \textit{of}.  Since its inception in the early 2000's, Heegaard Floer theory has given rise to invariants for a large range of mathematical objects; this one is distinguished not simply by its input, but also by the fact that the algebraic invariant behaves well under a natural topological operation.  

\subsection{Contact structures}

Recall from the introduction that a contact structure is a nowhere-integrable two-plane field.  We will consider contact structures only on orientable three-manifolds, and we further require that contact structures be \textit{cooriented}.  That is, each contact plane is oriented, so there is a consistent choice of positive normal vector. It will be useful to reference a coordinate model, so we introduce the \textit{standard contact structure} on $\mathbb{R}^3$, where the contact plane at each point is the kernel of the one-form $dz-y dx$.  (A cooriented contact structure may always be described as the kernel of such a \emph{contact form}.) In this case, the vector field $\partial_z$ coorients the contact planes.  We are primarily interested in studying contact manifolds up to \textit{contactomorphism}, that is, up to diffeomorphism preserving the plane fields.

Like topological manifolds, contact manifolds are locally simple but globally complicated.  The contact Darboux Theorem states that every point in a contact three-manifold has a neighborhood contactomorphic to a neighborhood of the origin in the standard contact $\mathbb{R}^3$.  In fact, some higher dimensional substructures also have well behaved neighborhoods.  For example, a curve segment everywhere transverse to the contact planes has a neighborhood contactomorphic to a neighborhood of the $z$-axis in the standard $\mathbb{R}^3$.    In this paper, we will focus  on the kind of two-dimensional submanifolds with particularly nice neighborhoods: convex surfaces. We will characterize convex surfaces by considering certain foliations they carry, but since codimension-one foliations are so central to the rest of the paper, we briefly detour into some general discussion before returning specifically to foliations on convex surfaces.   \

\subsection{Foliations on surfaces}\label{sec:fol} 

Throughout this article we will consider only oriented singular foliations  whose singularities are isolated and are either 
elliptic (see bottom right picture of Figure \ref{fig:braidfoliation}) or hyperbolic (see bottom left picture of Figure \ref{fig:braidfoliation}).  We denote  the set of elliptic singularities by $E$, and the set of hyperbolic singularities by $H$.  Unless explicitly noted, we assume that all regular leaves of the foliation compactify to oriented intervals. Elliptic points are either \emph{sources}, in which case they are also called \emph{positive} elliptic points, or \emph{sinks}, which are also called  \emph{negative} elliptic points. At a four-pronged hyperbolic singularity, the two opposite prongs oriented towards the hyperbolic point  form the \textit{stable separatrix}, while the two prongs oriented away from the hyperbolic point form  the \textit{unstable separatrix}.  The topological type of these foliations can be described combinatorially by the embedded graph formed by the stable and unstable separatrices of the hyperbolic points.  

The foliations appearing in this paper will have additional structure given by assigning a sign to each singular point. The signs of elliptic points have already been introduced, but the signs of hyperbolic points are not visible from the combinatorics of the foliation.  (The sign of a hyperbolic point comes from the orientation of the surface and additional local data that depends on the source of the foliation; see Sections~\ref{sec:charfol} and \ref{ssec:ob}.)  A foliation with the properties above, together with the extra partition of $H=H_+\sqcup H_-$ is called a \emph{signed singular foliation}; in the following we refer to oriented signed singular foliations simply as foliations.

Given a foliation $\fol$ on a surface $F$ satisfying  the hypotheses imposed above, we say that a multicurve $\Gamma\subset F$ is a \textit{dividing set} if $\Gamma$ is everywhere transverse to the leaves of $\mathcal{F}$ and  separates $F$ into two subsurfaces, each of which contains all the singularities of a fixed sign.   With this structure in hand, we are ready to introduce the characteristic foliation on a surface in a contact manifold, which is the key to the local neighborhood theorem mentioned above.  We introduce the aspects of this theory that we will need, and we recommend \cite{Mass} for further reading on the topic.

\subsection{Convex surfaces}\label{sec:charfol}

An oriented surface $F$ embedded in a contact three-manifold $(M, \xi)$ inherits a \emph{characteristic foliation} from $\xi$.  Intersecting the contact plane with the tangent plane at each point in the surface defines a line field, and the leaves of the characteristic foliation are the integral curves of these intersections.  
Characteristic foliations may be more general than the foliations described above, admitting leaves that are circles or even non-manifolds.  However, we will not consider any cases where these phenomena arise. 
 The orientation of the leaves follows from the coorientation of the contact structure, while the signs of the singular points depend on whether the coorientation of the contact structure is a positive or negative normal for $F$.

A surface in a contact structure is \emph{convex} if the contact structure is $I$-invariant in some product neighborhood;  a key result states that a surface is convex if and only if its characteristic foliation admits a dividing curve $\Gamma$  \cite{g-cor}. Remarkably,  the local neighborhood of a convex surface is determined by the dividing set alone.  The property of admitting a dividing curve (and hence, convexity) can be checked combinatorially, and in fact, convex surfaces are $C^\infty$-generic \cite{g-cor}.  Another important aspect of this equivalence is \emph{Giroux's flexibility}, it describes the sense in which $\Gamma$ captures  the essential data of the contact structure in a neighborhood of $F$.  Specifically, if $\Gamma$ is a  dividing set on $F\subset (M,  \xi)$, then  any  foliation divided by $\Gamma$ can be realized as the characteristic foliation of some isotopic surface $F'$ in a neighborhood of $F$.  Thus, given a  separating multicurve $\Gamma$ on a surface $F$, one may choose any foliation divided by $\Gamma$ and construct a compatible contact structure on $F\times I$.   If we choose another foliation, then  Giroux's flexibility implies this is the characteristic foliation on some surface inside this neighborhood, so our original neighborhood in fact contains the contact structure determined  by this new foliation.

\begin{figure}[h]
\centering
\includegraphics[scale=1]{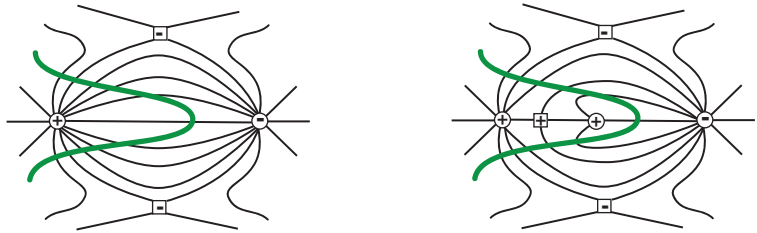}
         \caption{Local pictures of two characteristic foliations divided by the same curve $\Gamma$, shown in green.  Circles are elliptic points and the squares are  hyperbolic points.}\label{fig:diffol}
\end{figure}

Characteristic foliations may exhibit many leaf types, but we will restrict attention to the cases where the hypotheses of Section~\ref{sec:fol} are satisfied; this is also a generic property.  In addition, we will require that each signed singular foliation has no closed leaves or leaves connecting two hyperbolic points, and any such foliation will admit a dividing set, thus ensuring convexity.  To see this, we introduce two graphs $G_\pm$ embedded into $F$ and associated to $\mathcal{F}$. The vertices of $G_+$ are the positive elliptic points, and the edges between them are  the stable separatrices of positive hyperbolic points. The graph $G_-$ is analogously defined using the negative elliptic points and unstable separatrices. Observe that $G_+$ and $G_-$ are disjoint and that the complement of their neighborhoods $N(G_+)$ and $N(G_-)$ has no singularities and is thus foliated by intervals. The \emph{dividing curve} $\Gamma$ of such a foliation is given by the oriented boundary of $N(G_+)$, which is isotopic through curves transverse to the foliation to $-\partial N(G_-)$. When  a  foliation admits a dividing set, it is unique up to isotopy, so we will often refer to ``the'' dividing set.

With the given restrictions on leaf types (i.e., only intervals and leaves containing a single hyperbolic point), the complement of the union of the separatrices is a collection of disks, each with one positive and one negative elliptic point on its boundary. The interior of each disk is  foliated by an $I$-family of leaves from the positive to the negative elliptic point; this can be seen in Figure~\ref{fig:diffol}.


\section{Foliated open books}
\label{sec:fob}


 
 We saw in Section~\ref{sec:prelim} that the dividing set on a convex surface suffices to determine the contact structure in a neighborhood of the surface.  Although the precise information about the characteristic foliation is lost, enough data is retained to identify the relevant equivalence class.  This theme is pervasive throughout contact geometry, with open books being one of the most notable illustrations.  An open book decomposition of a contact manifold loses information about the specific contact structure, but with the benefit that the isotopy class of the contact structure is determined by a minimal amount of data.  This economical encoding of the isotopy class was first studied in the contact setting by Thurston-Winkelnkemper and rose in prominence with the work of Giroux \cite{TW} \cite{Gir}.  After recalling the classical construction, we will describe the \textit{foliated open books} first introduced in \cite{LV} as a new version of open books for contact manifolds with convex boundary.  Although the definition of a foliated open book will require us to keep track of more data on the boundary than simply the dividing set, the payoff will be a more user-friendly set of gluing theorems than seen with previous types of open books.

 
 An \emph{abstract open book} for a closed three-manifold is a pair $(S, h)$, where $S$ is a surface with boundary and $h$ an element of its mapping class group. This data suffices to construct an $S$-bundle over $S^1$, and after collapsing the boundary in a controlled way, yields a closed three-manifold.   A foliated open book adapts this approach to the setting of a manifold with boundary.  This time, the data consist of a sequence of $2k$ topologically distinct surfaces and the maps identifying one surface with the next.  Analogously, this determines a manifold with foliated boundary.  
 
  \subsection{Classical open books}
\label{ssec:ob}
This section reviews the definition of an open book decomposition of a closed three-manifold, along with the notion of an \textit{open book foliation} developed in \cite{IK1}.

 \begin{definition}\label{def:aob} An \emph{abstract open book} is a pair $(S, h)$ where $S$ is a surface with 
boundary $\partial S=B$ and  $h \colon S\rightarrow S$ is a diffeomorphism that preserves $B$ pointwise. 
\end{definition}

An abstract open book determines a closed three-manifold $M$ as follows.  First, consider the product $S\times I$ and identify the points $\big(h(x), 0\big)\sim (x,1)$ to form the mapping torus of $h$.  Then  collapse each component of the boundary $\partial S \times S^1$ to a circle via $(x,t)\sim (x,t')$ whenever $x\in \bdy S$.  The image of $\partial S\times S^1$ is an oriented link called the \textit{binding} and again denoted by $B$, while the surfaces $S\times \{t\}$ become the \textit{pages}.  We will also make use of the function $\pi: M\setminus B\rightarrow S^1$ that sends each point on $S\times \{t\}$ to $t$.

The simplest example of an open book is given by setting $S=D^2$, so that $h$ is necessarily isotopic to the identity.  The pair $(D^2, \id)$ determines $S^3$; to see this, observe that $N(B)$ and $M\setminus N(B)$ give a genus-one Heegaard splitting with meridional curves on the two solid tori intersecting once.  In fact, an open book determines not only a topological three-manifold, but actually a contact three-manifold, but this will be explored in the next section.  For now, we consider further topological structure associated to an open book.

Suppose that $M$ is the closed three-manifold built from $(S,h)$.  Then the pages of $S$ induce a foliation on a generic surface embedded in $M$.  Assume that a surface $F$ is transverse to the binding $B$, so that  $E=B\cap F$ is finite. Additionally, assume that the restriction $\widetilde{\pi}=\pi\vert_F\colon F\setminus E\to S^1$ is an $S^1$-valued Morse function with only one critical point for any critical value. Then the \emph{open book foliation} $\mathcal{F}_\pi$ on $F$ is the foliation induced by the level sets of $\widetilde{\pi}$ together with the elliptic points $E$. Equivalently, the leaves of the foliation are the intersections of $F$ with the pages of the open book.  As seen in  Figure \ref{fig:braidfoliation}, such a foliation may have three types of singularities: the points in $E$ are elliptic points; the index $0$ and $2$ critical points of $\widetilde{\pi}$ are centers; and the index $1$ critical points are  hyperbolic points. Each level set of $\widetilde{\pi}$ has at most one critical point, and there are no leaves connecting hyperbolic points. Although closed leaves may arise,  one may eliminate  them via a bigger isotopy of the surface \cite{IK1}. 

\begin{figure}[h]
\begin{center}
	\labellist
		\pinlabel $\bullet$ at 120 318
		\pinlabel $\bullet$ at 510 300
		\pinlabel $\bullet$ at 793 310
	\endlabellist
\includegraphics[scale=0.3]{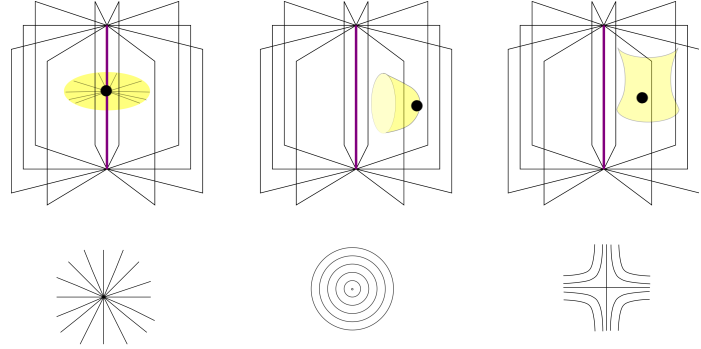}
\caption{The intersection of $F$ with the pages and binding (above) induces the singularity of $\mathcal{F}_\pi$ (below).  Left: the foliation on a disk transverse to the binding has an elliptic point.  Center: the foliation on a cup with one point tangent to a page has a center.  Right: the foliation on a saddle with one point tangent to a page has a hyperbolic point.}\label{fig:braidfoliation}
\end{center}
\end{figure}

As above, we can associate signs to the elliptic points depending on whether the binding coorients $F$ or not, whereas the sign of a critical point of $\widetilde{\pi}$ is given by the sign of $d\pi$ evaluated on the normal to $F$.  Just as characteristic foliations on convex surfaces determine the nearby contact structure, open book foliations determine the open book decomposition near the surface.


  \subsection{Foliated open books}
\label{ssec:fob}
Intuitively, a foliated open book is the structure on a manifold with boundary formed by cutting a classical open book along a surface with an open book foliation. We consider two examples of this sort before carefully stating a definition in parallel to Definition~\ref{def:aob}.

\begin{example}\label{ex:orange}

Consider the open book for $S^3$ described above with  connected binding and disk pages.  Choose a neighborhood of a point on the binding and cut $S^3$ along the boundary of this ball as shown in the center of Figure~\ref{fig:orange}. Discarding the complement of this ball, one sees that it inherits a binding and pages from the original open book, and that the new boundary is naturally equipped with the foliation whose leaves are boundary intervals of the pages.  This is the simplest possible foliated open book. 

For an example that is one step more interesting, cut $S^3$ along a pair of parallel spheres to get a thickened sphere that intersects the binding in two intervals.  The complement of these binding intervals is a union of rectangles.  

We will see more interesting examples after the formal definition.

\begin{figure}[h]
	\begin{center}
		\includegraphics[scale=0.6]{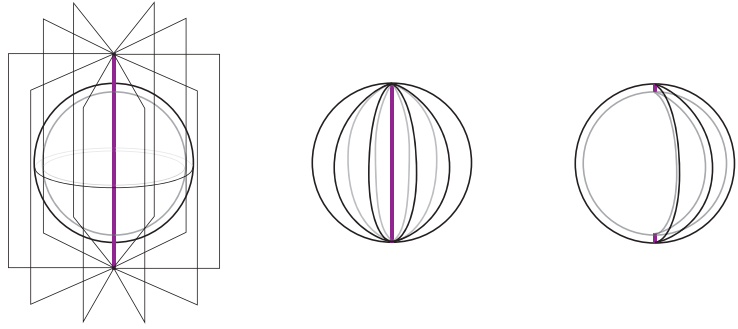}
		\caption{Cutting the open book $(D^2, \id)$ for $S^3$ (left) along a pair of parallel spheres yields a (pair of) foliated book(s) for the three-ball(s) (center) and a foliated open book for the thickened sphere (right, selected pages shown).  On the boundary spheres of the resulting foliated open books, each leaf of the open book foliation is a line of longitude, and the only singularities are the two elliptic points at the poles.}\label{fig:orange}
	\end{center}
\end{figure}

\end{example}

 \begin{definition}\cite[cf. Definition 3.12]{LV}\label{def:afob} An \emph{abstract foliated open book} is a tuple $(\{S_i\}_{i=0}^{2k}, h)$ where $S_i$ is a surface with 
boundary $\partial S_i=B\cup A_i$ \footnote{By a slight abuse of notation we denote the ``constant" part of the boundary of $S_i$ by $B$ for all $i$.} and corners at $E=B\cap A_i$ such that 
\begin{enumerate}
\item for all $i$, $A_i$ is a finite union of intervals and $B$ is a union of intervals or circles;
\item the surface $S_{i}$ is obtained from $S_{i-1}$ by either 
\begin{itemize}
\item attaching a 1-handle along two points on $A_{i-1}$, or 
\item cutting $S_{i-1}$ along a properly embedded arc $\gamma_{{i}}$ with endpoints in $A_{{i-1}}$ and then smoothing.\footnote{The indices of $\gamma_i$ in this paper are shifted compared to \cite{LV}, where the cutting arcs were denoted by $\gamma_{i-1}$.}
\end{itemize}
\end{enumerate}

Furthermore,  $h \colon S_{2k}\rightarrow S_0$ is a diffeomorphism between cornered surfaces that
preserves $B$ pointwise. 
\end{definition}

We invite the reader to pause and compare Definitions~\ref{def:aob} and \ref{def:afob}.  The latter has  two levels of complexity not seen in the classical definition: first, the definition replaces a single surface $S$ with a family  of  surfaces $S_i$ of distinct topological type, and second, the boundary of each surface is partitioned into $A_i$-intervals and $B$-intervals or -circles.  This second feature was seen already in Example~\ref{ex:orange}: cutting each page in the open book for $S^3$  along its intersection with the sphere resulted in two new bigon pages each bounded by an $A_i$-interval and a $B$-interval.  

\begin{example}
To illustrate the differences between classical and foliated open books, we consider a further example built by cutting the standard open book for $S^3$ along a separating $S^2$; See Figure~\ref{fig:benttorus}.   Here, the intersections between the indicated ball and the pages of the original open book are not all homeomorphic.  The points on the embedded $S^2$ where the changes in topological type occur are labeled by squares on the figure; the right-hand side of the figure shows the distinct subsurfaces (the pages of the resulting abstract open book), labeled to match the (embedded) pages in the original open book.

\begin{figure}[h]
	\begin{center}
		\includegraphics[scale=0.8]{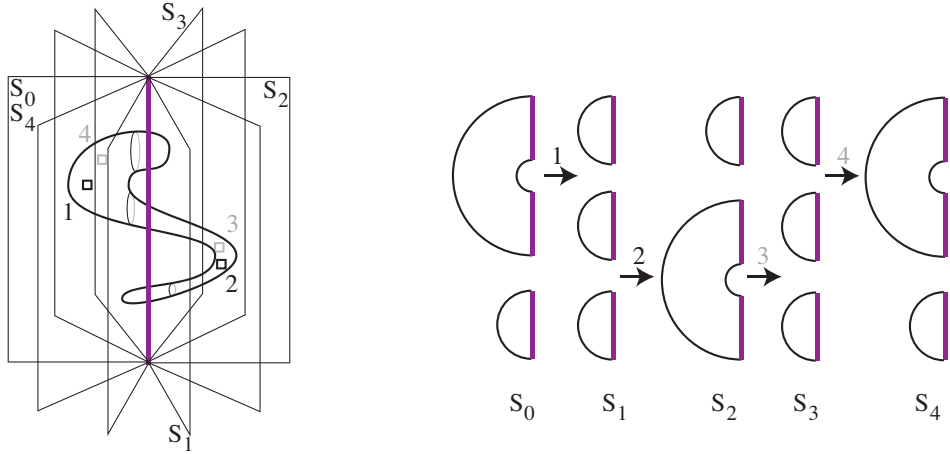}
		\caption{ A different foliated open book for a ball cut from $S^3$. }\label{fig:benttorus}
	\end{center}
\end{figure}
\end{example}

We now take on the full complexity of Definition~\ref{def:afob} and describe how to build a manifold from a sequence of pages of distinct topological types.   Throughout, we will use subscript indices to distinguish topologicially distinct page types, referring to these as ``abstract pages'' for convenience.

Any pair of consecutively indexed abstract pages $S_i$ and $S_{i+1}$ defines an elementary cobordism.  We build an analogue of the mapping torus by concatenating these elementary cobordisms and gluing $S_{2k}$ to $S_0$ by the map $h$.   More precisely, each abstract page $S_i$  yields a product $S_i\times \big(\frac{ i}{2k}, \frac{ i+1}{2k}\big)$, for $0\leq i\leq 2k-1$, and consecutive products join smoothly along a singular page which is surface with two points on its boundary identified. (Since $h:S_{2k}\rightarrow S_{0}$ is  a diffeomorphism,  we need not assign a separate product to $S_{2k}$).   After collapsing $B\times S^1$ to a multicurve called the \emph{binding} and still denoted by $B$, the remaining boundary is decorated by the non-binding boundaries of the pages.  With the above decomposition in mind, define a  function $\pi\colon M\setminus B\to S^1$ so that on each piece $S_i\times \big(\frac{ i}{2k}, \frac{ i+1}{2k}\big)$,  the function is projection to the second coordinate; here $S^1$ is identified with $[0,1]/(0\sim 1)$. Below, we abuse notation a couple of times and write $S_t$ for $\pi^{-1}(t)$.

This construction induces a singular foliation $\mathcal{F}$ on $\partial M$ whose regular leaves are copies of $A_i$, oriented as the boundary of the pages, and whose singular leaves each contain a single four-pronged hyperbolic point.  Equivalently, leaves are level sets of the restriction of the function $\pi$ to $\partial S$.  The elliptic points $E$ and the hyperbolic points $H$ each come with signs: each interval component of $A_i$ is oriented from a positive elliptic point towards a negative one.  Hyperbolic points associated to attaching a one-handle are negative, while hyperbolic points associated to cutting along an arc are positive; for an illustration of the latter, refer to Figure~\ref{fig:add}.

We denote the resulting smooth object by the triple $(M, \partial M, \mathcal{F})$ and call it a \emph{manifold with foliated boundary}.  We remember the identification of leaves with intervals on the boundary of abstract pages, and, in particular, the foliation has a distinguished union of $0$-leaves, which are always regular. Because there are no closed leaves or saddle-saddle connections, we may use the signs of the singular points to associate a  dividing set to the foliation: as seen in Section~\ref{sec:charfol}, the boundary of a neighborhood of the positive separatrices of positive hyperbolic points is a dividing set, and this is  unique up to isotopy transverse to the leaves.  Note that a manifold with foliated boundary does not have an associated foliated open book structure; rather, it has a boundary foliation that is compatible with the existence of a foliated open book.

We conclude with one more topological example before turning attention to the relationship between open books and contact structures. 


\begin{example}  \label{ex:fobfromfol} For a final example in this section, we describe a process for promoting a nicely foliated surface $F$ to a foliated open book $F\times I$ with the property that the open book foliation on each $F\times \{s\}$ is isotopic to the original foliation.  This procedure is described in detail in \cite[Section 4.2]{LV}, but we summarise it here for later use in this article. 

The open book decomposition near a surface $F$ is completely determined by the open book foliation $\mathcal{F}_\pi$ on $F$ \cite[Corollary 4.6]{LV}. In the following,  we describe this local structure by constructing a foliated open book for $F\times [-1,1]$ that embeds into any other (foliated) open book inducing $\mathcal{F}_\pi$.   Naively, one might try to cross the original surface with $[-1,1]$ and take the pages to be the products of leaves with the interval.  This works in the case of a foliation with only elliptic singularities, as in Example~\ref{ex:orange}, but the process is more subtle in the case that the original foliation has hyperbolic points.  

We first briefly describe the open book determined by $\mathcal{F}_\pi$ near $F$ before using the foliation to construct its abstract pages. The binding of the open book is transverse to $F$, so we can assume it embeds as $E\times I$ in $F\times I$, oriented  by $\frac{\partial}{\partial s}$ (respectively, $-\frac{\partial}{\partial s}$) for positive (negative) elliptic points. Recall that $\Gamma$ denotes the dividing set for a signed foliation.  Away from a neighborhood of $\Gamma\times I$,  each page $S_t$ is the union  of the leaves  $\widetilde\pi^{-1}(t\pm\varepsilon s) \times \{s\}\subset F\times\{s\}$, where the sign depends on whether we are in $F_+$ or $F_-$, and $\varepsilon$ is sufficiently small so that no page contains more than one hyperbolic point.  We connect these across $\Gamma\times I$ by bands which twist to compensate for the shearing of leaves in opposite directions as $|s|$ increases.  (Figure 8 in \cite{LV} provides local models  for this construction near $\Gamma$ and $E$.)

As noted above, when $t$ is not near a singular point, this yields pages which are simply thickened copies of the original interval leaves; when $F$ is closed, these are rectangular pages with two binding intervals separating a pair of leaves, one on each of $F\times\{\pm 1\}$.  This is illustrated by the thickened sphere in Example~\ref{ex:orange}.  

To see what happens near a singular value $t_0$ for the original foliation, consider the page which contains the corresponding hyperbolic point on $F\times\{0\}$.  The boundary of this page on $F\times\{-1\}$ is a copy of the $\pit^{-1}(t_0-\epsilon)$ leaf in which the saddle resolution has not yet happened, while the boundary of this page on $F\times \{1\}$ is a copy  of the $\pit^{-1}(t_0+\epsilon)$ leaf where the saddle resolution has already occurred.  This gives a recipe for writing down abstract pages: starting from the regular value $0$, set $S_0= \pit^{-1}(0)\times I$.  To form $S_1$, perform the first cut/add operation on the corresponding $F\times \{1\}$ edges of $S_0$; to form $S_2$, perform the corresponding add/cut operation on the $F\times \{-1\}$ edges of $S_1$. Note that $S_2$ can be thought of as $\pit^{-1}(t_0+\epsilon)\times I$, where $t_0$ is the first singular value encountered after $0$. We can continue to obtain a pair of pages for each hyperbolic point in the same way. If the original foliation had $n$ hyperbolic points, the new foliated open book will  therefore have $2n+1$ pages.  Each even-indexed page is a thickened regular leaf, while odd-indexed pages interpolate between these. Finally, note that the monodromy $h$ will always be trivial, as the first and last pages  are simply unions of disks.

\end{example}

\subsection{Foliated open books and contact structures}\label{sec:fobcontact}
With the topological constructions well in hand, we are ready to recall the compatibility between foliated open books and contact structures.

 \begin{definition}\label{def:supp} \cite[Definition 3.7]{LV} The abstract foliated open book $(\{S_i\}, h)$ \emph{supports} the contact structure $\xi$ on $(M, \partial M, \mathcal{F})$ if  \begin{enumerate}
 \item $TB$ is positively transverse to $\xi$; 
 \item there exists a nowhere zero vector field everywhere transverse to the interior of each page and to $\xi$ whose flow preserves $\xi$;
  \item there is a topological isotopy of $\partial M$ taking $\mathcal{F}$ to the characteristic foliation $\mathcal{F}_{\xi}$ such that some $\Gamma$ is a dividing set for each foliation throughout the isotopy. 
 \end{enumerate}
 \end{definition}

 
 We will often want to consider a manifold with foliated boundary $(M, \partial M, \mathcal{F})$  together with a contact structure $\xi$ supported by a foliated open book inducing the boundary foliation; we call this a 
  \emph{foliated contact three-manifold} and denote it by the triple $(M, \xi, \fol)$.
  
  As above, we may ignore the third condition to recover the classical definition of a contact structure supported by an open book decomposition of a closed manifold.  
  If a three-manifold $M$ has both an open book decomposition $(B,\pi)$ and a contact structure $\xi$ supported by this open book, then a sufficiently generic surface will carry both a characteristic foliation $\mathcal{F}_\pi$  and an open book foliation $\mathcal{F}_\xi$. A priori these foliations are unrelated, but if the open book foliation has no circle leaves, then  the contact structure
  can be isotoped so that the characteristic foliation and the open book foliation have the same combinatorics and further, that the singular points agree  \cite{IK1}.  This is the key observation that gives the boundary  criteria for foliated open books.

In the examples produced by cutting an honest open book along a separating surface, observe that the open book foliations on the two new boundaries match, but with the signs of all singular points reversed.  Conversely, any two foliated open books with matching, sign-reversed boundary foliations may be glued to produce a closed manifold with an open book structure.  In fact, these cutting and gluing results respect the contact structures supported by each of the open books in the sense of Definition~\ref{def:supp} \cite[Theorems 6.1, 6.2]{LV}.   
  
In the remainder of this section, we construct several additional foliated open books for specific contact manifolds.  Example~\ref{ex:s3-hopf} constructs foliated open books for a pair of distinct contact structures on the three-ball.  In this case, as in the examples above, the foliated open books are identified as submanifolds of an open book for a closed three-manifold.  Finally, Example~\ref{ex:ot-sfob} is a specific  instance of the procedure described in Example~\ref{ex:fobfromfol} above; we endeavor to provide a plausible construction here while referring the reader to \cite{LV} for the technical details.
  
  \begin{example}\label{ex:s3-hopf} Different open book decompositions of a fixed topological manifold may determine different contact structures, and  the same holds true in the case of foliated open books.  In this example we consider a pair of open books for $S^3$, one of which supports the unique tight contact structure and the other of which  supports an  overtwisted contact structure.  Cutting each of these along a separating $S^2$ yields foliated open books for tight and overtwisted balls. 
    
Let $(A,  h^\pm)$ denote the  open book  for $S^3$ with annular  pages and monodromy a single Dehn twist of the indicated sign. The binding of the associated open book decomposition is a positive (resp.\ negative) Hopf link, denoted by $H^+$ (resp.\ $H^-$).  To  picture this, consider the genus one Heegaard splitting of $S^3=H_1\cup_{\bdy}(-H_2)$ into two solid tori where $H^+$ (resp. $H^-$) is embedded on the Heegaard torus as in Figure~\ref{fig:HopfOB}.  Here $\pi^{-1}([0,\frac{1}{2}])=H_1$ and $\pi^{-1}([\frac{1}{2},1])=H_2$. The positive twist monodromy  induces the tight contact structure on $S^3$, while the negative twist monodromy induces an overtwisted structure.  

\begin{figure}[h]
	\begin{center}
		\includegraphics[scale=0.6]{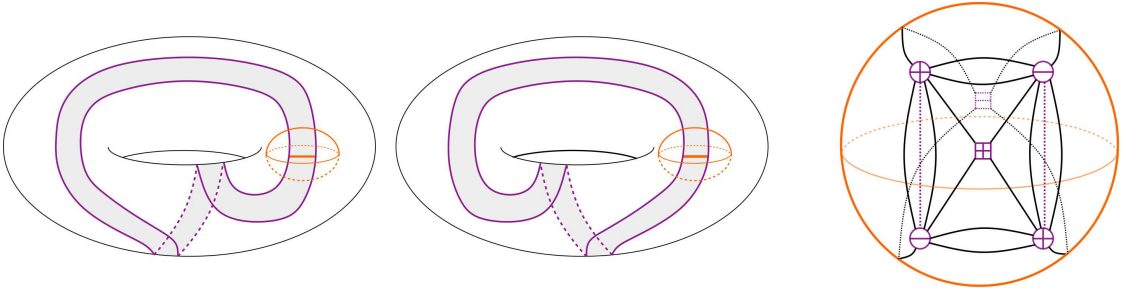}
		\caption{ Left: $H^+$ embedded in $\partial H_1$ in a  tight $S^3$. Center: $H^-$ embedded in $\partial H_1$ in an  overtwisted $S^3$.  Right: the open book foliation on the boundary of a neighborhood of the spanning arc in the shaded annulus; as in Figure~\ref{fig:diffol}, elliptic points are drawn as circles, and hyperbolic ones as squares.}\label{fig:HopfOB}
	\end{center}
\end{figure}

In each of  these  open books, consider the embedded $S^2$ bounding  a neighborhood of the orange arc in $\pi^{-1}(\frac{1}{2})$ shown in Figure \ref{fig:HopfOB}.   Discard this ball, leaving a pair of foliated open books for the complementary tight and overtwisted balls.  The open book foliation on the boundary sphere has four elliptic points and two hyperbolic points as in the right hand picture of Figure \ref{fig:HopfOB}.  The pages of these foliated open books are shown as the shaded subsurfaces in Figure ~\ref{fig:3ball-h}.   The Dehn twists from the original open book restrict to Dehn twists on the annular pages of the foliated open books.


\begin{figure}[h]
	\begin{center}
	\labellist
		\pinlabel $S_0$ at 70 -15
		\pinlabel $S_1$ at 270 -15
		\pinlabel $S_2$ at 470 -15
	\endlabellist
		\includegraphics[scale=0.62]{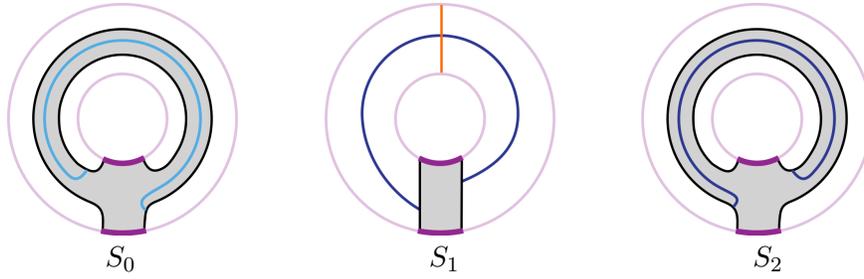}
		\vspace{.3cm}
		\caption{Removing a neighborhood of the orange arc from $S^3$ yields the shaded pages for a foliated open book for the three-ball.  Each abstract page is shown  embedded into an annular page for the open book $(A, h^\pm)$,  where $h^\pm$ is a positive (resp.\ negative) Dehn twist around the core of the annulus.  These twists restrict to the cornered annulus $S_2$ as the monodromy for the foliated open books $(\{S_0, S_1, S_2\}, h^+)$ for a tight three-ball and $(\{S_0, S_1, S_2\}, h^-)$ for an overtwisted three-ball.  The light and  dark blue curves are sorting arcs, which are introduced in Section~\ref{ssec:sortfob}.}\label{fig:3ball-h}
	\end{center}
\end{figure}
\end{example}




\begin{example}  \label{ex:ot-sfob} In this example we construct a foliated open book for a ball supporting an overtwisted contact structure.  This example is borrowed from \cite{LV2}, following the procedure  summarized in Example~\ref{ex:fobfromfol}.  The motivation for including this intially-opaque construction is that it will allow us to characterise any foliated open book for an overtwisted contact manifold in terms of a particular embedded foliated open book.  To begin this process, we introduce a non-standard definition of overtwistedness:

\begin{definition}\label{def:ot} A contact manifold $(M, \xi)$ is \emph{overtwisted} if it contains an embedded disk whose boundary is everywhere transverse to  $\xi$ and whose characteristic foliation is as shown in Figure~\ref{fig:ottopbot}.
\end{definition}

\begin{figure}[h]
		\begin{center}
			\includegraphics[scale=0.8]{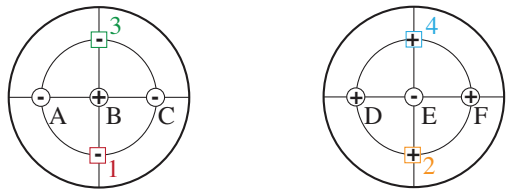}
			\caption{The left-hand picture shows the top of an overtwisted disk with transverse boundary; elliptic and hyperbolic points are again drawn as circles and squares, respectively. The right-hand picture labels the singularities of the characteristic foliation on the underside of the disk; the two points in the pairs $(A,D)$, $(B,E)$, and $(C,F)$ coincide,  but the sign of each singular point is reversed when viewed from the opposite side.}\label{fig:ottopbot} 
		\end{center}
	\end{figure}

Overtwistedness is more commonly characterized in terms of the existence of an embedded disk with a different characteristic foliation, but it's a consequence of  Giroux Flexibility  that the existence of a disk with this foliation is equivalent to the existence of disks with related characteristic foliations. 
We choose Definition~\ref{def:ot} with  a later application in mind.  We now apply the construction sketched in Example~\ref{ex:fobfromfol} to build a foliated open book for a neighborhood of this disk; it follows that inside any overtwisted contact manifold, we may find an overtwisted ball that admits this foliated open book. 

 The existence of a transverse boundary requires us to slightly modify the construction, smoothing the boundary of pages associated to leaves that terminate on $\partial F$.  Thus a regular leaf connecting an elliptic point $e$ to  $\partial F$ will give rise to a bigon with one  $e \times [-1,1]$ component and one $A_i$ component connecting $e \times\{\pm1\}$.

We now apply this construction to the overtwisted disk shown in Figure~\ref{fig:ottopbot}, yielding an abstract foliated open book with five abstract pages.  We set $t=0$ to consist of the leaves where intervals connect (1) elliptic points   $A$ and $B$, and (2) elliptic point $C$ to the boundary. The first leaf becomes a rectangular page with two boundary intervals, one connecting $A\times\{1\}$ and $B\times\{1\}$ and the other connecting $D\times \{-1\}$ with $E\times \{-1\}$.  The leaf connecting $C$ to the boundary becomes a bigon whose boundary interval connects $C \times \{1\}$ with $F \times \{-1\}$.  See Figure~\ref{fig:stab1-1}. Around a positive elliptic point, $t$ increases in the positive direction; following the procedure outlined in Example~\ref{ex:fobfromfol}, the first hyperbolic singularity corresponds to adding a handle to connect these two pages.  Figrue~\ref{fig:stab1-1} shows all the abstract pages of the resulting foliated open book.

	\begin{figure}[h]
		\begin{center}
			\includegraphics[scale=0.8]{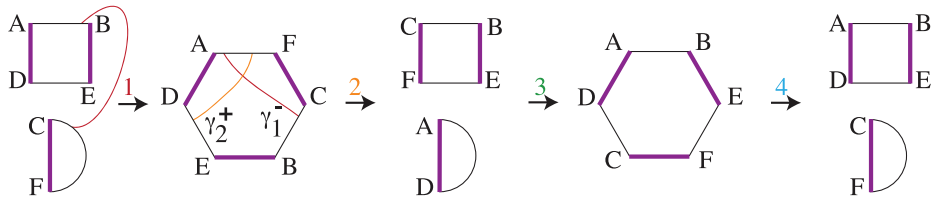}
			\caption{The pages of a foliated open book for a neighborhood of the disk in Figure~\ref{fig:ottopbot}.  Each hyperbolic point in the original foliation induces a pair of  hyperbolic points of opposite sign in the foliated open book.  (The labeled arcs are explained in Example~\ref{ex:ot-sfobstab}.)}\label{fig:stab1-1}
		\end{center}
	\end{figure}

	Since each component of each page is a disk, there is a unique (up to isotopy) way to identify successive pages, and the foliated open book is completely determined by this data. One may also reconstruct the dividing set on the ball.  One component encircles $B$ on the ``top'' of the ball, while two further components bound an annulus containing $D$, $F$, and the two positive hyperbolic points on the ``bottom'' of the ball.  In contrast, the foliated open book for the overtwisted ball constructed in Example~\ref{ex:ot-sfob} has a connected dividing set.
	
To illustrate how this ball might embed in an overtwisted contact manifold, we consider the open book for an overtwisted $S^3$ from Example~\ref{ex:s3-hopf}.  Recall that the pages are annuli and the monodromy is a left-handed Dehn twist.  The top half of Figure~\ref{fig:embotdisc} shows a ball intersecting two representative pages of this open book.  The elliptic points are labeled to identify these subsurfaces with the first and third abstract pages from Figure~\ref{fig:stab1-1}; although we find it difficult to visualise further pages embedded in $S^3$, it is not difficult to embed the foliated open book pages in abstract pages, as shown below.

			\begin{figure}[h]
		\begin{center}
			\includegraphics[scale=0.8]{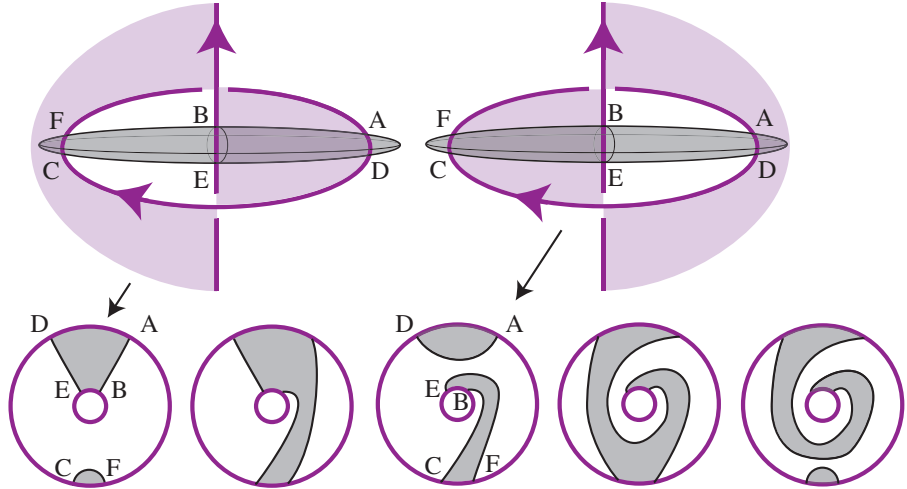}
			\caption{The foliated open book for the minimal overtwisted ball embeds in the simplest  open book for an overtwisted $S^3$.  The monodromy is a left-handed Dehn twist.}\label{fig:embotdisc}
		\end{center}
	\end{figure}	
\end{example}

\subsection{Sorted foliated open books}\label{ssec:sortfob}
Foliated open books will be our means to associating a Floer-theoretic invariant to a three-manifold with foliated boundary.  However, in order to generate a multipointed Heegaard diagram, we will need to require the further technical condition that our foliated open book is \emph{sorted}.
 Since the notation to verify this condition is somewhat involved, we pause to motivate it first.

 The definition of a foliated open book  requires successive pages to evolve by cutting or by gluing, but we may equivalently think of this as the condition that evolution is always by addition, but in either direction: either $S_i$ is obtained from $S_{i-1}$ by a one-handle addition or else $S_{i}$ is obtained from $S_{i+1}$ by a one-handle addition.   One-handles associated to negative hyperbolic points are those already described in Definition~\ref{def:afob} as ``adds",  while positive hyperbolic points correspond to adding a handle as the page index decreases.  We will call a foliated open book \textit{sorted} if a one-handle, after being added with respect to some direction (i.e., increasing or decreasing indices), persists for all subsequent pages with respect to that direction.  See Figure~\ref{fig:add}.

\begin{figure}[h!]
		\begin{center}
			\includegraphics[scale=0.63]{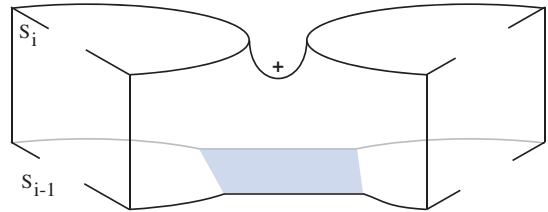}
			\caption{ Here $S_{i-1}$ is obtained from $S_i$ by adding the shaded one-handle, inducing a positive hyperbolic singularity at the saddle point. The sorted condition requires that this handle persist for all $S_j$ with $0\leq j<i$.  Note that the binding has been blown up as $B\times I$ for ease of viewing.}\label{fig:add}
		\end{center}
	\end{figure}
	
 Recall that the elliptic points $E=A_i\cap B$ partition as $E=E_+\sqcup E_-$, where each interval is oriented from a point  $e_+\in E_+$ to a point $e_-\in E_-$. We impose the following conventions on the cutting and gluing arcs that govern how the pages evolve:

\begin{itemize}
\item If $S_i\rightarrow S_{i+1}$ cuts $S_i$ along  a properly embedded arc, the endpoints of the arc  lie near the $e_+$ end of the intervals of $A_i$.  We decorate $S_i$ and all prior pages with a copy of the cutting arc and label these arcs as $\gamma^+_i$.  If $S_j$ is decorated with multiple cutting arcs  near the same point $e_+$, the indices decrease with the orientation of $A_j$.
\item If $S_i\rightarrow S_{i+1}$ adds a one-handle to $S_i$, the points of the attaching sphere separate any $\gamma^+$ endpoints from the $e^-$ on the intervals of $A_i$.  We decorate $S_{i+1}$ and all subsequent pages with  a copy of the cocore of the attached one-handle and label these arcs as $\gamma^-_i$.  If $S_j$ is decorated with multiple cocores  near the same point $e_-$, the indices decrease with the orientation of $A_j$.
\end{itemize}
 
 If we take the perspective that gluing is simply cutting in with the order of the indices reversed, then the second bullet point can be phrased in identical language to the first.  Figure~\ref{fig:sortconv} illustrates these conventions in an example.
 
\begin{definition}\label{def:sort}
 A foliated open book is \emph{sorted} if the arcs $\gamma^-\cup \gamma^+$ are mutually disjoint on all the pages where they appear.  We denote a sorted foliated open book by $(\{S_i\}_{i=0}^{2k},h, \{\gamma_i^\pm\})$.
\end{definition}

A foliated open book which is sorted has a ghost page: a minimal surface formed by cutting along all of the arcs.  Although this surface may not actually coincide with any $S_i$ in the foliated open book, it embeds as a  subsurface in each abstract page. Keeping  this in mind may help in understanding the following notation-heavy definition.

Suppose $(\{S_i\}_{i=0}^{2k},h,\{\gamma_i^\pm\})$ is a sorted foliated open book for foliated contact three-manifold $(M, \xi, \fol)$.
On each page $S_i$, let $P_i$ be the complement of a ``cornered'' neighborhood of $A_i \cup(\bigcup_{i < j}\gamma_j^+)\cup(\bigcup_{i\geq j}\gamma_j^-)$, with corners at $E$.  This $P_i$ is the ghost page and exists as subsurface of each $S_i$.  The copies of $P_i$ may be identified via the flow of a vector field transverse to the pages, and we denote the composition of these identifications from $P_0\subset S_0$ onto $P_{2k}\subset S_{2k}$ by $\iota$.

\begin{figure}[h]
		\begin{center}
			\includegraphics[scale=0.79]{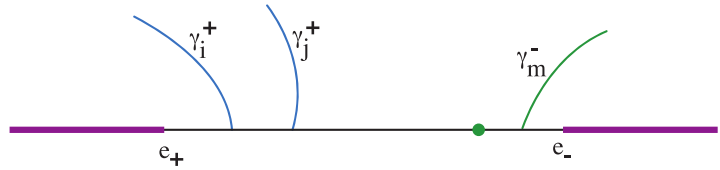}
			\caption{ An indicative interval of $A_n$.  Here $i>j\geq n >m$.  The arcs $\gamma_i^+$ and $\gamma_j^+$ show arcs that will be cut along on higher-index pages.  The bold dot indicates where a one-handle could be attached on some later page, while the arc $\gamma_m^-$ is the cocore of a handle already attached. }\label{fig:sortconv}
		\end{center}
	\end{figure}

\subsection{Sorting by stabilization}
\label{ssec:sfob}
In this subsection we examine the operation of positive stabilization on a foliated open book and show how it can be used to render a non-sorted foliated open book sorted.  The idea is straightforward: each stabilization adds a one-handle to every page of the foliated open book by taking a connect sum with a foliated open book for the standard tight $S^3$.  For a simple example, we note that the foliated open book for a tight three-ball constructed in Example~\ref{ex:s3-hopf} is a positive stabilization of the foliated open book from Example~\ref{ex:orange}.

The number of sorting arcs $\gamma^{\pm}$ is controlled by the foliation, and hence unchanged by stabilization.  Repeating the process sufficiently many times gives the sorting arcs  more space in the enlarged page to avoid each other.   Of course, the arcs that guide the stabilization must be chosen carefully, and we explain how to do this below.  The formal proof that this is always possible may be found in \cite{LV}.  

As shown in \cite{LV}, stabilization may be understood as a concrete example of gluing two foliated open books.  Choose an arc $(\gamma, \partial \gamma)$  embedded  in a fixed page $(S_t, B)$ of a foliated open book.    A regular neighborhood of this arc may be chosen so that its boundary is a sphere whose signed singular foliation has two hyperbolic singularities.   Choosing such neighborhoods in two separate foliated open books yields manifolds with matching foliated boundaries. Since we can only glue foliations where the singularities match, but with opposite signs, shifting the $t$ coordinate by $\frac{1}{2}$ allows us to glue the two spheres to construct a foliated open book for the connect sum of the two original manifolds; the new pages are Murasugi sum of the pages of the original foliated open books. If one of the manifolds was an open book with annular pages supporting the tight contact structure on $S^3$, then the contactomorphism type of the manifold is unchanged and we say that the foliated open book has been \emph{positively stabilized}. The open book in Example~\ref{ex:s3-hopf} with positive Hopf twist binding is a stabilization of the elementary  open book for $S^3$ from Example~\ref{ex:orange}.

The description above applies with minor modification to all versions of open books, but a distinguishing feature of foliated open books is the non-homogeneity of the pages.  An arc on $S_t$ may not persist to some later page $S_{t'}$, or $S_{t'}$ may have a non-trivial mapping class group even though $S_t$ was a collection of disks.  This highlights that there are two choices to be made when defining a stabilization of a foliated open book: which page, and which arc?

With a goal of removing intersections of the form $\gamma_i^+\cap \gamma_j^-$, choose a regular page between the hyperbolic points $h_i^+$ and $h_j^-$.  We will stabilize along an arc in this page so that as $\gamma_i^+$ rises up through the manifold, the subinterval that would collide with the descending $\gamma_j^-$ picks of the monodromy of the foliated open book for $S^3$ and instead undergoes a Dehn twist around the core of the annular Murasugi summand of the page.  

\begin{example}  \label{ex:ot-sfobstab}
Example~\ref{ex:ot-sfob} introduced a foliated open book for an overtwisted ball which embeds into any overtwisted contact manifold.  Examining Figure~\ref{fig:stab1-1} more closely will show that it is not sorted, and this example will perform the sorting stabilizations.
	
	The first hyperbolic singularity is negative and corresponds to adding a one-handle to $S_0$ as shown; on the second page $S_1$, the cocore of the one-handle is recorded as an arc $\gamma_1^-$.  However, the second hyperbolic singularity is positive and corresponds to cutting the second page along the arc labelled $\gamma_2^+$ to get the third page.  As shown in the figure,  $\gamma_1^-$\ and $ \gamma_2^+$ intersect.
	
To remove this obstruction to sortedness, choose a copy of $S_1$ and stabilize along an arc that crosses $\gamma_2^+$ and $\gamma_1^-$ once.  The result is shown in Figure~\ref{fig:stab2-1}.  One can think of $\gamma_1^-$ as undergoing a right-handed twist as it ascends  or  $\gamma_2^+$ as undergoing a left-handed twist as it descends, and since the two curves now avoid each other, we may proceed with increasing $t$ until  $\gamma_3^-$ and $\gamma_4^+$ intersect on the new $S_3$ page.  
	
\begin{figure}[h]
		\begin{center}
			\includegraphics[scale=0.79]{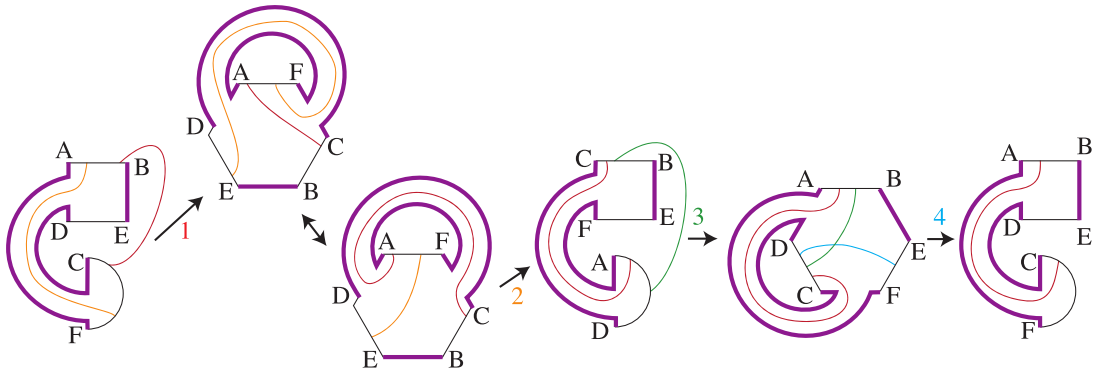}
			\caption{ The stabilization of the foliated open book from Figure~\ref{fig:stab1-1}.  Note that $\gamma_3^+$ and $\gamma_4^-$ intersect on the new page $S_3$, so the foliated open book remains unsorted. }\label{fig:stab2-1}
		\end{center}
	\end{figure}
	
	To remove the intersection $\gamma_3^-\cap \gamma_4^+$, we analogously stabilize along an arc intersecting each of these curves once.  Finally, a sorted foliated open book is seen in Figure~\ref{fig:stab2}.  Since the gluing map   is inherited from the original foliated open book, it remains translation in the page as drawn. 
	
\begin{figure}[h]
		\begin{center}
			\includegraphics[scale=0.79]{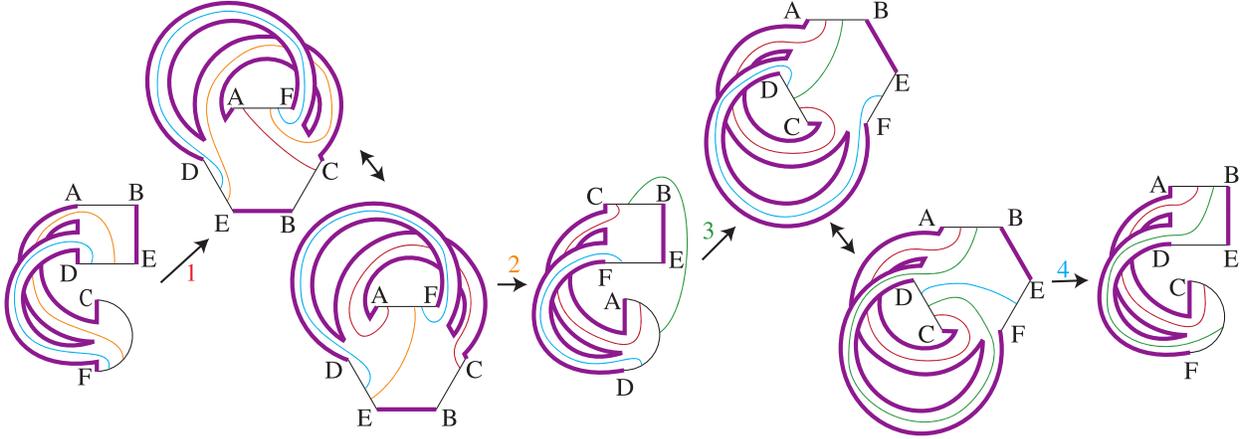}
			\caption{  A sorted foliated open book for a neighborhood of an overtwisted disk, obtained from the foliated open book in Figure~\ref{fig:stab1-1} by a sequence of two stabilizations.   }\label{fig:stab2}
		\end{center}
	\end{figure}
	
	For any $i$, cutting along all the sorting arcs on $S_i$ yields a pair of disks, the ``ghost page''  described in the previous section.
	\end{example}


\section{Multipointed bordered Floer homology}
\label{sec:mpbordered}

As a stepping stone for defining link Floer homology, Ozsv\'{a}th and Szab\'{o} defined a \emph{multipointed} version of $\hfhat$ denoted by $\hftilde$ \cite{oszlink}. This version is defined using Heegaard diagrams with multiple basepoints, and, given a closed, oriented three-manifold $M$, it is related to $\hfhat(M)$ by the isomorphism
\[\hftilde(M,n)\cong\hfhat(M)\otimes V^{n-1}.\]
Here, $n$ is the number of basepoints and $V$ is a $2$-dimensional graded $\Z/2\Z$-vector space with  generators in gradings $1$ and $0$; i.e., $V\cong H_{\ast}(S^1)$.

In \cite{sfh}, Juh\'asz defines an extension of $\hfhat$ for non-closed three-manifolds whose boundary is \emph{sutured}, called sutured Floer homology.
Note that both $\hfhat(M)$ and $\hftilde(M, n)$ are sutured Floer homologies of specific sutured manifolds corresponding to $M$. Specifically, 
let $M(n)$ be the sutured manifold obtained from $M$ by removing $n$ pairwise disjoint balls and adding as a suture one oriented simple closed curve on each resulting sphere boundary component. Then, we have $\hfhat(M)\cong \sfh(M(1))$ while $\hftilde(M,n)\cong \sfh(M(n))$.

Lipshitz, Ozsv\'{a}th, and Thurston define bordered Floer homology as an extension of $\hfhat$ for three-manifolds with parametrized boundary \cite{LOTbook}. First, they associate a differential graded algebra  $\cala(\bdy M)$ to the parametrization. Then, they define an $\ainf$-module, or \emph{type $A$ structure},  $\cfahat(M)$ over $\cala(\bdy M)$, or equivalently, a \emph{type $D$ structure} (roughly,  a dg module) $\cfdhat(M)$ over $\cala(-\bdy M)$. These invariants are constructed to satisfy a nice gluing formula which recovers $\hfhat$. Specifically, if $M$ is a closed three-manifold obtained by a gluing $M_1\cup_{\bdy}M_2$, then the derived tensor product $\cfahat(M_1)\widetilde{\otimes}_{\cala(\bdy M_1)} \cfdhat(M_2)$  (which often has a smaller model denoted $\boxtimes$) is homotopy equivalent to $\cfhat(M)$.

A generalization of bordered Floer homology, called bordered sutured Floer homology, was defined by Zarev \cite{bs}. It is an invariant of three-manifolds whose boundary is ``part sutured, part parametrized". This invariant satisfies a  gluing formula which recovers sutured Floer homology.

In this section, we introduce a \emph{multipointed} theory for bordered Floer homology as a special case of bordered sutured Floer homology.  
First, we recall the definition of the boundary parametrization in bordered Floer homology. Let $M$ be a three-manifold with boundary of genus $k$. A parametrization for $\bdy M$ 
 consists of a disk $D\subset\bdy M$; a basepoint $z\in\bdy D$; and $2k$ pairwise disjoint properly embedded arcs $\amalg_{i=1}^{2k}\delta_i$ in $\bdy M\setminus \Int(D)$ such that $M\setminus\left(D\cup \amalg_{i=1}^{2k}\delta_i\right)$ is an open disk.  The parametrization data is recorded by a pointed matched circle $\zz=(Z,a,m)$, where $Z=\bdy D$ with $z\in Z$, $a=\bdy (\amalg_{i=1}^{2k}\delta_i)$ union of $4k$ points on $Z$, and $M$ is a matching on $a$ that pairs endpoints of the same arc $\delta_i$. 
 
 \begin{definition}
A \emph{pointed matched multicircle} is a triple $\zz=(Z,a,m)$ where $Z=\amalg_{i=1}^{n}Z_i$ is a union of $n$ circles with a basepoint $z_i$ on each $Z_i$, $a\subset Z$ is a set of an even number of points, and $m\colon a\to a$ is a matching. Given a three-manifold $M$ with boundary of genus $k$, a \emph{(multipointed) parametrization} of $\bdy M$  is a pointed matched multicircle $\zz$ with $|a|=4n+4k-4$, along with  an embedding of $Z$ and of pairwise disjoint arcs $\delta=\amalg_{i=1}^{2n+2k-2}\delta_i$ into $M$, satisfying the following:
\begin{enumerate}
 \item The image of each $Z_i$ bounds a disk $D_i$ in $\bdy M$ whose interior is disjoint from the arcs $\delta_i$ for all $i$.
 \item  $\bdy \delta = a$ and each $\bdy \delta_i$ is a pair of points matched by $m$.
 \item $\bdy M\setminus\left((\amalg_{i=1}^n D_i)\cup(\amalg_{i=1}^{2n+2k-2}\delta_i)\right)$ is the union of $n$ open disks such that each disk contains exactly one of the marked points $z_i$ for $i=1,\ldots, n$.
 \end{enumerate}
 We call the three-manifold with  multipointed parametrized boundary a \emph{bordered manifold}, as in \cite{LOTbook},  and denote it by $(M, \zz)$, omitting from the notation the implicit data of how the arcs $\delta_i$ are embedded on $\bdy M$.
 \end{definition}

A three-manifold with  multipointed parametrized boundary $(M, \zz)$ can be reinterpreted as a bordered sutured manifold $(M,  \bsGamma,\zz^{\circ})$ where $\amalg _{i=1}^n D_i$ is the sutured part while its complement is the parametrized part, and $\zz^{\circ}$ is the arc diagram obtained from $\zz$ by removing neighborhoods of the basepoints.  Thus, Zarev's construction associates a type $A$ structure  $\bsahat(M, \bsGamma,\zz^{\circ})$ over  $\cala(\zz)\coloneqq \cala(\zz^{\circ})$, or equivalently a type $D$ structure $\bsdhat(M, \bsGamma,\zz^{\circ})$ over  $\cala(-\zz)$. The construction uses a Heegaard diagram presentation $\HD = (\Sigma, \alphas, \betas, \zz^{\circ})$ for the bordered sutured manifold. The arc diagram $\zz^{\circ}$ is embedded on $\bdy\HD$ so that there is one interval on each component of $\bdy\HD$.  The structures $\bsahat$ and $\bsdhat$ are generated by certain sets of intersection points in $\alphas\cap \betas$ on the Heegaard diagram and they have structure maps defined by counting certain holomorphic curves in $\Sigma\times I\times \R$ whose projection onto $\Sigma$ avoids the regions of $\Sigma\setminus (\alphas\cup \betas)$ containing $\bdy\HD\setminus \zz^{\circ}$.

The embedding of  $\zz^{\circ}$ on $\bdy\HD$  can be extended to an identification of $\zz$ with $\bdy\HD$, by reinserting the basepoints, one in each component of $\bdy\HD\setminus \zz^{\circ}$. The result is a \emph{multipointed bordered Heegaard diagram} for $(M,\zz)$. Since there is no loss of information when moving from one perspective to the other, we denote  $\zz^{\circ}$ simply by $\zz$ in this paper.  We will denote the structures $\bsahat(M, \bsGamma,\zz^{\circ})$ and $\bsdhat(M, \bsGamma,\zz^{\circ})$ by $\cfat(M, \zz)$ and $\cfdt(M, \zz)$, respectively. Explicitly, given a multipointed bordered Heegaard diagram, these structures are defined by counting the ``usual" holomorphic curves; the condition of ``avoiding the basepoints" is equivalent to ``avoiding $\bdy\HD\setminus \zz^{\circ}$".  The gluing formula for bordered sutured Floer homology implies that if the closed three-manifold $M$ with multiple basepoints is obtained by gluing  multipointed bordered three-manifolds $M_1\cup_{\bdy} M_2$, with $M_1$ parametrized by $\zz$ and $M_2$ by $-\zz$, then $\cfat(M_1)\boxtimes_{\cala(\zz)} \cfdt(M_2)$ is homotopy equivalent to $\cftilde(M,n)$.


\section{The bordered contact invariant}
\label{sec:hd}

Let $(M,\xi, \fol)$ be a foliated contact three-manifold. In \cite{afhlpv}, a  sorted foliated open book for $(M,\xi, \fol)$ was used to construct a Heegaard diagram for an associated bordered sutured manifold $(M,  \bsGamma,\mathcal{Z})$, along with a preferred generator of the diagram. The homotopy equivalence class of this generator in the resulting  bordered sutured Floer homology is an invariant of the foliated contact three-manifold \cite[Theorem 1]{afhlpv}.  In particular, the class is independent of the choice of open book. We recall the construction next, slightly rephrasing to use multipointed bordered Floer homology, and we work out a small example.

As explained in Section~\ref{sec:mpbordered}, we can convert the data of a bordered sutured manifold  $(M,  \bsGamma,\mathcal{Z})$ to multipointed bordered data for a simpler perspective.  We describe the parametrization on the boundary of the resulting bordered manifold $(M, \zz)$ directly below.

We use the foliation to define a natural parametrization of $\bdy M$ via a pointed matched multicircle $\zz=(Z,a,m)$. 
Recall that the data of the foliation remembers the page index associated to each leaf, and in particular, that there is a distinguished union of regular leaves denoted by $A_0$. 
Let $D\subset \bdy M$ be a closed neighborhood of $A_0$, and let $Z = \bdy D$. Note that $D$ is a union of $n$ disks, where $2n$ is the number of elliptic points in the foliation. Let $\delta_i$ be subarc of the positive (resp. negative) separatrix for $h_i^{+}$ (resp. $h_i^-$) that lies in $\bdy M\setminus (\mathrm{int} D)$.  Define  $a\subset Z$ to be the set of points that are the boundaries of $\delta_i$ and let $m$ be the matching induced on the points in $a$ by $\delta_i$. For each component of $Z$, mark a basepoint with a smallest possible  $(0,2\pi)$-coordinate.
See Figure~\ref{fig:bord-fol-mfd} for an example. It is easy to check that $\zz=(Z,a,m)$ together with the embedding of the arcs $\delta_i$ parametrizes $\bdy M$.

\begin{figure}[h]
	\begin{center}
		\labellist
		\pinlabel $\textcolor{Turquoise}{\delta_1}$ at 100 150
		\pinlabel $\textcolor{blue}{\delta_2}$ at 90 56
	\endlabellist
		\includegraphics[scale=0.69]{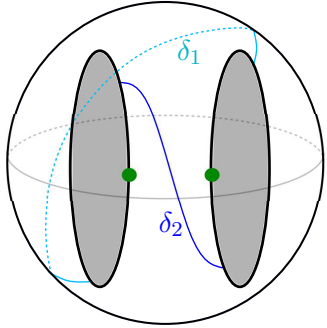}
		\caption{The bordered three-manifold associated to the foliated ball that is the neighborhood of the orange arc from Example~\ref{ex:s3-hopf}. The two grey disks make up $D$, their boundary is $Z$, the two acs $\delta_i$ are drawn in blue and labelled on the figure, and the basepoints are drawn in green.}\label{fig:bord-fol-mfd}
	\end{center}
\end{figure}

Now, fix an abstract sorted foliated open book $(\{S_i\}_{i=0}^{2k},h, \{\gamma_i^\pm\})$ for  the foliated contact three-manifold $(M,\xi,\mathcal{F})$. The sortedness condition ensures that the first page of the open book, together with its (indexed) $\gamma_i^+$ arcs; the last  page together with its (indexed) $\gamma_j^-$ arcs; and the monodromy $h$  fully describe the manifold. In fact, the union of the first and last page naturally describes a (cornered) handlebody decomposition for $M$. Using the data of $(\{S_i\},h, \{\gamma_i^\pm\})$, we describe a multipointed bordered Heegaard diagram $\HH = (\Sigma, \alphas,\betas, \zz)$ for this handlebody decomposition, along with a preferred generator. We outline the construction below; cf.~\cite[Section 3]{afhlpv}.

 Let $g_i$ be the genus of $S_i$ and let $n_i$ be the number of boundary components of $S_i$. Recall that  the boundary of the cornered surface $S_i$ is $B\cup A_i$, where $B$ is a union of circles and arcs, and $A_i$ is a union of intervals only.

We let $\Sigma=S_{0}\cup_{B}-S_{0}$.  In order to distinguish the two copies, we will write 
\[\Sigma=S_{\epsilon}\cup_{B}-S_{0},\] 
but we emphasize that  $S_\epsilon$ can be identified with $S_0$. The surface $\Sigma$ has genus $2g_0+n_0-1$ and $|A_0|$ boundary components. 

For $i\in H_-$, consider the $S_{2k}$ copies of the sorting arcs $\gamma_i^-$, and let $\beta^-_i = -h(\gamma_i^-)$ on $-S_0$.
 For $i\in H_+$, consider  the $S_{\epsilon}$ copies of the sorting arcs $\gamma_i^+$.  The endpoints of $\gamma_i^+$ lie  near the $E_+$ end of intervals of $A_\epsilon$. Isotope the arcs $\{\gamma_i^+\}$ (simultaneously, to preserve disjointness) near the endpoints along $-\bdy\Sigma$ until they all lie in $I_+\subset A_0$;  the isotopy stops after crossing $E_+$ and before encountering $\cup_{j\in H_-}-h(\gamma_j^-)\subset -S_0$. Call the resulting arcs $\beta_i^+$. 
Define a set of arcs $\betas^a = \{\beta_1^a, \ldots, \beta_{2k}^a\}$ by
\[\beta_i^a = \begin{cases}
\beta_i^+  & \textrm{if $i\in H_+$,} \\
\beta_i^- & \textrm{if $i\in H_-$.}
\end{cases}\]
As in \cite{afhlpv}, we use the notation $\beta_i^a$  or $\beta_i^{\pm}$ if $i\in H_{\pm}$ interchangeably.

 Let ${\boldsymbol b} = \{b_1, \ldots, b_{2g_0+n_0+|A_0|-k-2}\}$ be a set of cutting arcs for $P_{\epsilon}\subset S_{\epsilon}$ disjoint from $\betas^a$ and with endpoints on $B$, so that each connected component of $S_{\epsilon}\setminus ({\boldsymbol b}\cup \betas^a)$ is a disk with exactly one interval of $A_{\epsilon}$ on its boundary. (In \cite{afhlpv}, we show this can always be achieved.)  In other words, ${\boldsymbol b}$ is a basis for $H_1(P_\varepsilon, B)$. Recalling the identification  $S_\epsilon=S_0$, we may push $b_i\subset S_0$ through $M$ to lie on $S_0$ again and define
   \[\beta_i=b_i\cup -h\circ \iota(b_i)\subset S_{\epsilon}\cup_{B}-S_{0},\]
   where $\iota$ is the identification of $P_0$ with $P_{2k}$ from  Section~\ref{ssec:sortfob}. 
Write $\betas^c=\{\beta_1, \dots, \beta_{2g_0+n_0+|A_0|-k-2}\}$.

 For each cutting arc $b_i\in \boldsymbol b$ on $S_\epsilon$, let $a_i$ be an isotopic curve formed by pushing the endpoints negatively along the boundary so that $a_i$ and $b_i$ intersect once transversely. Similarly, for each arc $b_i^+ \coloneqq S_{\epsilon}\cap \beta_j^+$, let $\tilde a_j$ be an isotopic curve formed by pushing the endpoints negatively along the boundary so that $\tilde a_j$ and $b_i^+$ (and equivalently $\tilde a_j$ and $\beta_j^+$) intersect once transversely.  We ``double" each of these arcs to form the $\alpha$-circles which define the handlebody $S_0\times [0,\epsilon]$.  Namely, define  
  \begin{align*}
 \alpha_i &= a_i\cup -a_i \subset S_{\epsilon}\cup_{B}-S_{0}\\
 \widetilde{\alpha}_j &= \tilde a_j\cup -\tilde a_j \subset S_{\epsilon}\cup_{B}-S_{0},
 \end{align*}
   and write $\alphas^c = \{\widetilde\alpha_i\}_{i\in H_+}\cup \{\alpha_1, \dots, \alpha_{2g_0+n_0+|A_0|-k-2}\}$.
Place a basepoint on each interval of $A_{\epsilon}\subset S_{\epsilon}\subset \Sigma$. Write
\[{\bf z} = \{z_1, \ldots, z_{|A_{\epsilon}|}\}\]
for the set of basepoints. 

We say that a multipointed bordered Heegaard diagram $\HH = (\Sigma, {\alphas}, {\betas},\zz)$ constructed as above is \emph{adapted} to the sorted abstract foliated open book $(\{S_i\},h, \{\gamma_i^\pm\})$ and to the corresponding foliated contact three-manifold $(M,\xi,\mathcal{F})$.

Let $\HH$ be a  multipointed bordered  Heegaard diagram adapted to $(\{S_i\},h, \{\gamma_i^\pm\})$. In \cite{afhlpv}, we show that any such diagram is admissible. (In fact, in \cite{afhlpv} neighborhoods of basepoints are drilled out to obtain a bordered sutured diagram for a certain bordered sutured manifold naturally associated to $(M,\xi,\mathcal{F})$, but we suppress this discussion here.) 
 Using the  notation introduced above, 
define 
\[\xxx = \{x_1, \ldots, x_{2g_0+n_0+|A_0|-k-2}\} \cup \{x_i^+ \mid i\in H_+\}\]
to be the set of unique intersection points 
\begin{align*}
x_i &=  a_i\cap b_i \in S_{\epsilon}\subset \Sigma\\
x_i^+ &= \tilde a_i\cap b_i^+\in S_{\epsilon} \quad \textrm{ if } i\in H_+.
\end{align*}

We will use $\xxx$ to define two contact invariants in multipointed bordered Floer homology.

By \cite[Proposition 3.4]{afhlpv}  and \cite[Section 3.4]{bs-JG}, the diagram ${\overline{\HD}=(\Sigma, \betas,\alphas,\overline{\zz})}$ obtained by exchanging the roles of the two sets of curves and formally replacing the arc diagram $\zz$ of $\beta$-type (which is to say, parametrized by arcs which are part of the second set of curves)  with the identical arc diagram $\overline{\zz}$ of $\alpha$-type (parametrized by arcs which are part of the first set of curves) is a  multipointed bordered diagram for $(-M, \overline{\zz})$.  Write $\overline{\zz}=(Z, a, m)$. We have the following proposition.

\begin{proposition}[cf.~{\cite[Proposition 3.5]{afhlpv}}]\label{prop:xd-def}
The above $\xxx$ gives a well defined generator
\[
\xxx_D := \xxx\in \cfdt(\overline{\HH})
\]
with  $I_D(\xxx) = I(H_-)$ and $\delta^1(\xxx_D) = 0$, 
and a well defined generator
\[
\xxx_A  := \xxx\in \cfat(\overline{\HH})
\]
with $I_A(\xxx_A) = I(H_+)$ and $m_{i+1}(\xxx_A, a(\brho_1), \ldots, a(\brho_i))=0$ for all $i\geq 0$ and all sets of Reeb chords $\brho_j$ in $(Z, a)$. 
\end{proposition}

\begin{example}\label{ex:ball1}

We illustrate the construction outlined above using the (sorted) foliated open book in Figure~\ref{fig:3ball-h}. Recall that the three pages depicted in Figure~\ref{fig:3ball-h} in fact can be used to construct different foliated open books, depending on the choice of monodromy $\tau^{n}$, for $n\in \Z$, where $\tau$ is a positive Dehn twist along the core of the annular page $S_2$. 

First, consider the foliated open book with pages depicted in Figure~\ref{fig:3ball-h} and monodromy $\tau$ (this was denoted by $h^+$ in Figure ~\ref{fig:3ball-h}). Figure~\ref{fig:3ball-h-plus} shows  the associated  Heegaard diagram $\HD^+$.
\begin{figure}[h]
	\begin{center}
	\labellist
		\pinlabel $S_{\epsilon}$ at 60 -10
		\pinlabel $-S_{0}$ at 450 -10
		\pinlabel $\cup_B$ at 260 50
		\pinlabel $x_1$ at 97 40
		\pinlabel $y_1$ at 54 142
		\pinlabel $x_2$ at 379 43
		\pinlabel $y_2$ at 464 118
	\endlabellist
		\includegraphics[scale=0.7]{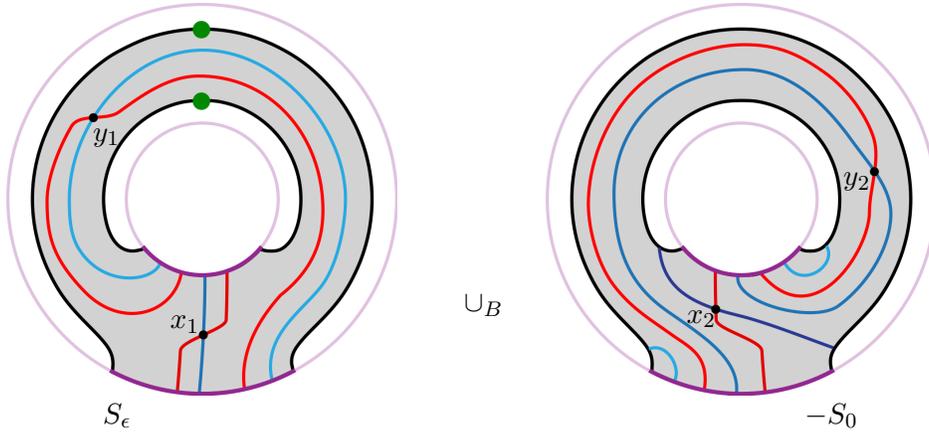}
		\vspace{.2cm}
		\caption{The Heegaard diagram for the sorted foliated open book $(\{S_0, S_1, S_2\}, h^+)$ from Figure \ref{fig:3ball-h}.  The monodromy $h^+$ is a positive Dehn twist, so the images $\beta_2^- = -h^+(\gamma_{2}^-)$ and $ -h^+\circ \iota(b_1)$ are  the dark and medium-dark blue curves  on $-S_0$, respectively. Intersection points are labelled differently from the above definition, for convenience. The contact generator $\xxx$ is the pair $\{x_1, y_1\}$, or $x_1y_1$ for short.}\label{fig:3ball-h-plus}
	\end{center}
\end{figure}
We label the intersection points in the Heegaard diagram $\HD^+$ by $x_1$, $x_2$, $y_1$, and $y_2$ as in Figure~\ref{fig:3ball-h-plus}. The diagram has two generators -- $x_1y_1$ and $x_2y_2$, where $x_1y_1$ is the special generator $\xxx$ defined above. Let $\rho_1$ and $\rho_2$  be the algebra elements in $\mathcal{A}(\bdy\overline{\HD^+})$ corresponding to the Reeb chords on the inside and outside boundary components of the Heegaard diagram, respectively, as seen on Figure~\ref{fig:3ball-h}. 
 The type $D$ structure $\cfdt(\overline{\HD^+})$ is generated by $x_1y_1$ and $x_2y_2$, and has structure maps
\begin{align*}
\delta^1(x_1y_1)  &= 0\\
\delta^1(x_2y_2)  &= (\rho_1+\rho_2)\otimes x_1y_1.
\end{align*}
The contact class $c_D(B^3,\xi, \fol)$ is the homotopy equivalence class of $x_1y_1$.

Next, consider the foliated open book with pages depicted in Figure~\ref{fig:3ball-h} and monodromy $\tau^{-1}$ (which was denoted $h^-$ in Figure ~\ref{fig:3ball-h}). Figure~\ref{fig:3ball-h-minus} shows  the associated  Heegaard diagram $\HD^-$.
\begin{figure}[h]
	\begin{center}
	\labellist
		\pinlabel $S_{\epsilon}$ at 60 -10
		\pinlabel $\cup_B$ at 260 50
		\pinlabel $-S_{0}$ at 450 -10
		\pinlabel $x_1'$ at 96 42
		\pinlabel $y_1'$ at 52 140
		\pinlabel $x_2'$ at 378 59
		\pinlabel $x_3'$ at 397 25
		\pinlabel $x_4'$ at 424 20
		\pinlabel $y_2'$ at 425 60
		\pinlabel $y_3'$ at 447 36
		\pinlabel $y_4'$ at 476 44
	\endlabellist
		\includegraphics[scale=0.7]{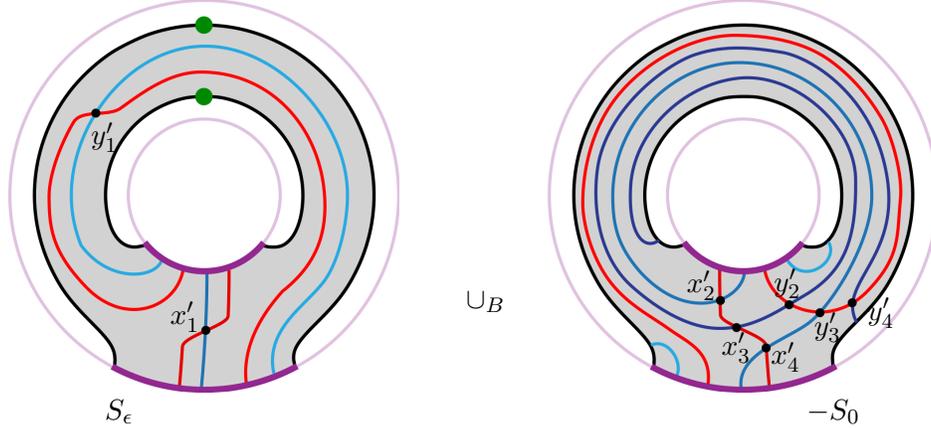}
		\vspace{.2cm}
		\caption{The Heegaard diagram for the sorted foliated open book $(\{S_0, S_1, S_2\}, h^-)$ from Figure \ref{fig:3ball-h}.  The monodromy $h^-$ is a negative Dehn twist, so the images $\beta_2^- = -h^-(\gamma_{2}^-)$ and $ -h^-\circ \iota(b_1)$ are  the dark and medium-dark blue curves  on $-S_0$, respectively. The contact generator $\xxx$ is $x_1'y_1'$.}\label{fig:3ball-h-minus}
	\end{center}
\end{figure}
We label the intersection points in $\HD^-$ by $x_1'$, $x_2'$, $x_3'$, $x_4'$, $y_1'$, $y_2'$, $y_3'$, and $y_4'$ as in Figure~\ref{fig:3ball-h-minus}. Let $\rho_1, \rho_2\in \mathcal{A}(\bdy\overline{\HD^-})$ be as in the previous example.
The type $D$ structure $\cfdt(\overline{\HD^-})$ is generated by $x_1'y_1'$, $x_1'y_2'$, $x_1'y_4'$, $x_2'y_1'$, $x_2'y_2'$, $x_2'y_4'$, $x_3'y_3'$, $x_4'y_1'$, $x_4'y_2'$, and $x_4'y_4'$, and has structure maps

\begin{align*}
\delta^1(x_1'y_1') &= 0\\
\delta^1(x_1'y_2') &= \rho_1\otimes x_1'y_1' \\
\delta^1(x_1'y_4') &= \rho_2\otimes x_1'y_1'\\
\delta^1(x_2'y_1') &= I\otimes x_1'y_1' \\
\delta^1(x_2'y_2') &= \rho_1\otimes x_2'y_1' + I\otimes x_1'y_2'\\
\delta^1(x_2'y_4') &= I\otimes x_1'y_4' + \rho_2\otimes x_2'y_1'\\
\delta^1(x_3'y_3') &= 0 \\
\delta^1(x_4'y_1') &= I\otimes x_1'y_1' \\
\delta^1(x_4'y_2') &= I\otimes x_1'y_2' +   I\otimes x_3'y_3' \\
\delta^1(x_4'y_4') &= I\otimes x_1'y_4' +   \rho_2\otimes x_4'y_1' 
\end{align*}
In particular, $\delta^1(x_2'y_1') = I\otimes x_1'y_1'$  implies that there is a type $D$ homotopy equivalence from $\cfdt(\overline{\HH^-})$  to an equivalent structure, carrying $x_1'y_1'$ to zero.

\end{example}


\section{Vanishing of the contact class for overtwisted structures -- a local argument}
\label{sec:local}

In this section, we illustrate the power of invariants compatible with cut-and-paste constructions by providing a local argument that the contact class $c(\xi)$ for closed contact manifolds vanishes  if the contact structure is overtwisted. 

We begin by showing that the bordered contact invariant vanishes for a neighborhood of an overtwisted disk. Specifically, we consider the foliated open book constructed in \cite{LV2} for a three-ball  neighborhood  $(B^3, \xi_{\mathrm{OT}}, \fol_{\mathrm{OT}})$.  In Example~\ref{ex:ot-sfob}, we stabilized the foliated open book from \cite{LV2} to a sorted one. 
We now construct the Heegaard diagram $\HH$ associated to the resulting sorted foliated open book from Figure \ref{fig:stab2}. 
For convenience, in Figure~\ref{fig:s0s4-OTdisk} we display again the pages $S_0$ and $-S_4$, along with the sorting arcs decorations. 

\begin{figure}[h]
	\begin{center}
		\includegraphics[scale=0.7]{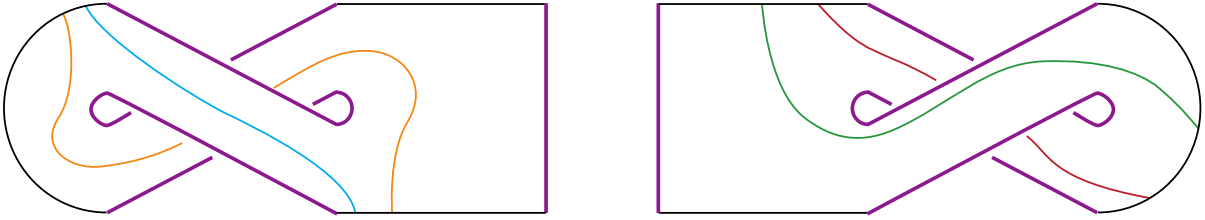}
		\caption{The first page (to the left) and the mirror of the last page (to the right) of the sorted foliated open book in Figure \ref{fig:stab2}.}\label{fig:s0s4-OTdisk}
	\end{center}
\end{figure}

Figure~\ref{fig:HD-OTdisk} shows  the associated  Heegaard diagram $\HH$.

\begin{figure}[h]
	\begin{center}
	\labellist
		\pinlabel $x_1$ at 38 80
		\pinlabel $y_1$ at 132 67
		\pinlabel $w_1$ at 232 107
		\pinlabel $w_4$ at 367 51
		\pinlabel $w_3$ at 402 92
		\pinlabel $w_2$ at 450 95
		\pinlabel $x_2$ at 480 31
		\pinlabel $x_3$ at 450 28
		\pinlabel $x_4$ at 418 65
		\pinlabel $y_2$ at 535 50
		\pinlabel $y_3$ at 555 75
		\pinlabel $x_5$ at 582 35
		\pinlabel $x_6$ at 608 89
		\pinlabel $x_7$ at 580 92
	\endlabellist
		\includegraphics[scale=0.7]{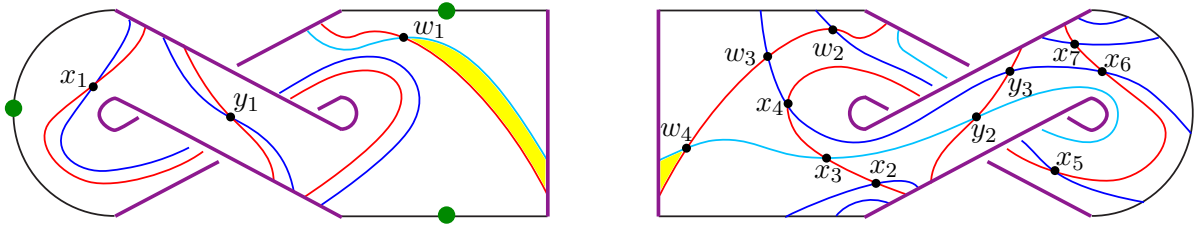}
		\caption{The Heegaard diagram for the sorted foliated open book in Figure \ref{fig:stab2}.  The monodromy $h$ is the identity, so the images $\beta_i^- = -h(\gamma_{i}^-)$ are simply the sorting arcs $\gamma_i^-$ on $-S_0$. Intersection points are labelled differently from the above definition, for convenience. The contact generator $\xxx$ is the triple $\{x_1, y_1, w_1\}$, or $x_1y_1w_1$ for short.}\label{fig:HD-OTdisk}
	\end{center}
\end{figure}

The generator $x_1y_1w_1\in\cfdt(\overline{\HH})$ represents the contact class. 
We claim that there is a unique holomorphic curve that avoids the basepoints and is asymptotic to  $x_1y_1w_4$ at $-\infty$, and this curve ends at $x_1y_1w_1$.  

Indeed, $x_1$ and $y_1$ cannot be starting moving coordinates for a holomorphic curve; the only non-basepointed regions at these intersection points are the thin strips supported on the $S_{\epsilon}$ part of the diagram, but the orientation on these strips is \emph{into} $x_1$ and $y_1$. So any holomorphic curve starting from $x_1y_1w_4$ must only have $w_4$ as a moving coordinate. A curve that hits the boundary of the Heegaard diagram would need to have a moving coordinate on a $\beta$-arc. Since $w_4$ is on a $\beta$-circle, all holomorphic curves starting from $x_1y_1w_4$ project to the interior of the diagram. Thus, any such curve with a single moving coordinate projects to an immersed bigon. By counting local coefficients, the yellow bigon from $x_1y_1w_4$ to $x_1y_1w_1$ in  Figure~\ref{fig:HD-OTdisk} represents the unique such curve.  

Thus, considering $\cfdt(\overline{\HH})$, we have $\delta^1(x_1y_1w_4) = I\otimes x_1y_1w_1$. Or, if one prefers to consider $\cfat(\overline{\HH})$, we have $m_1(x_1y_1w_4) = x_1y_1w_1$, whereas higher products $m_i$ vanish on $x_1y_1w_4$.  It follows that there is a type $D$ (resp.\ type $A$) homotopy equivalence from $\cfdt(\overline{\HH})$ (resp.\ $\cfat(\overline{\HH})$) to an equivalent structure, carrying $x_1y_1w_1$ to zero.

Recall from the introduction that we claimed the Ozsv\'ath--Szab\'o vanishing result for overtwisted contact manifolds can be recovered from gluing properties of the bordered contact invariant.  In fact, the necessary technical results have already been established, and we conclude by assembling them into the promised proof.

\begin{proof}[Proof of Corollary~\ref{cor:ot}]
Suppose $(M, \xi)$ is a closed overtwisted three-manifold. As discussed in Section~\ref{sec:fobcontact}, $(M, \xi)$  contains an overtwisted disk whose neighborhood is contactomorphic to the contact three-ball $(B^3, \xi_{\mathrm{OT}}, \fol_{\mathrm{OT}})$ studied in Example~\ref{ex:ot-sfob}. Thus, $(M, \xi)$ decomposes as the union of two foliated contact three-manifolds, one of which is $(B^3, \xi_{\mathrm{OT}}, \fol_{\mathrm{OT}})$.The computation, above, together with Theorem~\ref{thm:glue} and functoriality for $\boxtimes$, implies that $c(\xi)=0$. 
\end{proof}

\bibliographystyle{alpha}

\bibliography{master}

\end{document}